[1994/12/01]

\documentclass[draft]{amsart}

\usepackage{amsmath, amsfonts, amssymb}
\usepackage{amsthm,amscd,latexsym,mathrsfs,hyperref}
\usepackage{bm}

\theoremstyle{plain} 
\newtheorem{theorem}{Theorem}[section]
\newtheorem{corollary}[theorem]{Corollary}
\newtheorem{lemma}[theorem]{Lemma}
\newtheorem{proposition}[theorem]{Proposition}

\theoremstyle{definition}
\newtheorem{definition}{Definition}[section]

\DeclareMathOperator{\tr}{\mathrm{tr}}
\DeclareMathOperator{\supp}{\mathrm{supp}}
\DeclareMathOperator{\argmin}{\mathrm{arg\, min}}

\begin{document}

\title[Strong law of large numbers for the $L^1$-Karcher mean]{Strong law of large numbers for the $L^1$-Karcher mean}
\author[Yongdo Lim and Mikl\'os P\'alfia]{Yongdo Lim and Mikl\'os P\'alfia}
\address{Department of Mathematics, Sungkyunkwan University, Suwon 440-746, Korea.}
\email{ylim@skku.edu}
\address{Department of Mathematics, Sungkyunkwan University, Suwon 440-746, Korea and Bolyai Institute, Interdisciplinary Excellence Centre, University of Szeged, H-6720 Szeged, Hungary.}
\email{palfia.miklos@aut.bme.hu}


\date{\today}

\begin{abstract}
Sturm's strong law of large numbers in $\mathrm{CAT}(0)$ spaces has been an influential tool to study the geometric mean or also called Karcher barycenter of positive definite matrices. It provides an easily computable stochastic approximation based on inductive means. Convergence of a deterministic version of this approximation has been proved by Holbrook, providing his "nodice" theorem for the Karcher mean of positive definite matrices. The Karcher mean has also been extended to the infinite dimensional case of positive operators on a Hilbert space by Lawson-Lim and then to probability measures with bounded support by the second author, however the $\mathrm{CAT}(0)$ property of the space is lost and one defines the mean as the unique solution of a nonlinear operator equation on a convex Banach-Finsler manifold. The formulations of Sturm's strong law of large numbers and Holbrook's "nodice" approximation are natural and both conjectured to converge, however all previous techniques of their proofs break down, due to the Banach-Finsler nature of the space. In this paper we prove both conjectures by establishing the most general $L^1$-form of Sturm's strong law of large numbers and Holbrook's "nodice" theorem in the operator norm by developing a stochastic discrete-time resolvent flow for the Karcher barycenter using its Wasserstein contraction property.
\end{abstract}

\maketitle

\section{The contractive barycenter of positive operators}
Let $\mathbb{S}$ denote the vector space of self-adjoint operators equipped with the operator norm $\|\cdot\|$ on a Hilbert space $\mathcal{H}$ and let $\mathbb{P}\subset\mathbb{S}$ denote the cone of invertible positive definite operators. On $\mathbb{S}$ the closure $\overline{\mathbb{P}}$ of the cone generates the positive definite partial order $\leq$ also called the Loewner order. When $\mathcal{H}$ is finite dimensional, then $\mathbb{P}$ is the convex cone of positive definite matrices and it comes equipped with the natural trace metric
\begin{equation*}
\langle X,Y\rangle_A:=\tr\{A^{-1}XA^{-1}Y\}
\end{equation*}
for $A\in\mathbb{P}$ and $X,Y\in\mathbb{S}$. This Riemannian metric has its distance function of the form
\begin{equation*}
d(A,B)=\|\log(A^{-1/2}BA^{-1/2})\|_2
\end{equation*}
for $A,B\in\mathbb{P}$ and the Frobenius $2$-norm $\|X\|_2:=\sqrt{\tr\{X^*X\}}$. Means of elements of $\mathbb{P}$ and multiplicative ergodic theorems on it were studied in a large number of papers, see for example \cite{Antezana,bhatiaholbrook,bhatiakarandikar,Karlsson,lawsonlim,lawsonlim1,lawsonlim12,limpalfia,limpalfia2}. The metric space $(\mathbb{P},d)$ is a complete $\mathrm{CAT}(0)$ space, in particular it has nonpositive sectional curvature \cite{bhatiaholbrook}. This is equivalent to the $2$-convexity of $X\mapsto d^2(X,A)$, so a natural mean on the space is the \emph{Frechet} or \emph{Karcher} barycenter defined as
\begin{equation}\label{eq:KarcherMean}
\Lambda(\mu):=\argmin_{X\in\mathbb{P}}\int_{\mathbb{P}}d^2(X,A)-d^2(Y,A)d\mu(A)
\end{equation}
for any Borel probability measure $\mu$ that integrates the distance $d(X,Y)$ for a fixed (thus all) $Y\in\mathbb{P}$. This definition is natural from the point of view of $\mathrm{CAT}(0)$ spaces, thus adopted by Sturm \cite{sturm} in that setting, and he proved a nice generalization of the law of large numbers, which states the almost sure convergence of the stochastic inductive mean sequence $\{S_n\}_{n\in\mathbb{N}}$ to $\Lambda(\mu)$, where $S_n$ is defined recursively as $S_1:=Y_1$, 
\begin{equation}\label{eq:InductiveMeans}
S_{n+1}:=S_n\#_{\frac{1}{n+1}}Y_{n+1}
\end{equation}
for i.i.d. random variables $Y_n$ with law $\mu$ that has bounded support, where $t\in[0,1]\mapsto A\#_tB$ denotes the unique minimal geodesic connecting $A$ and $B$. Moreover Sturm also derived that the barycenter $\Lambda$ is $1$-Lipschitz with respect to the $L^1$-Wasserstein distance $W_1$.

It turns out that $\#_t$ admits a nice closed formula in the particular case of $\mathbb{P}$. We have
\begin{equation*}
A\#_tB=A^{1/2}\left(A^{-1/2}BA^{-1/2}\right)^tA^{1/2}=A\left(A^{-1}B\right)^t,
\end{equation*}
the weighted geometric mean of positive operators $A,B\in\mathbb{P}$, which is monotone \cite{bhatiaholbrook} with respect to the partial order $\leq$ generated by the cone $\mathbb{P}$. Thus Sturm's law of large numbers was used by Lawson-Lim \cite{lawsonlim} and Bhatia-Karandikar \cite{bhatiakarandikar} to prove the monotonicity of $\Lambda(\sum_{i=0}^{k-1}\frac{1}{k}\delta_{A_i})$ in $\leq$ with respect to the variables $A_i\in\mathbb{P}$, then an important conjecture in matrix analysis. Later a deterministic, also called "nodice", version of Sturm's law that periodically recycles all the points $A_i$ was proved by Holbrook \cite{Hol} in $\mathbb{P}$ and then in the $\mathrm{CAT}(0)$ metric setting by the authors \cite{limpalfia2}, and independently in \cite{Ba2}. This "nodice" theorem states that $S_n$ converges to $\Lambda(\sum_{i=0}^{k-1}\frac{1}{k}\delta_{A_i})$ for the recycling deterministic version $Y_n:=A_{\overline{n}}$ in \eqref{eq:InductiveMeans}, where $\overline{n}$ denotes the residual of $n$ modulo $k$. Further generalizations were proved in the $\mathrm{CAT}(1)$ setting by Ohta and the second author \cite{ohta0} and by Yokota \cite{yokota}. These metric approaches treat the sequence of inductive means $S_n$ as a discrete-time approximation of the gradient flow of the cost function to be minimized in \eqref{eq:KarcherMean}, and apply the Riemannian-like nature of the $\mathrm{CAT}(1)$ property in an essential way. For further results on the continuous time metric theory of gradient flows see \cite{AGS,bacak,ohta}.

It turns out that this metric formulation of the mean $\Lambda$ is no longer possible in the operator case when $\mathcal{H}$ is infinite dimensional. In that setting the available metrics on $\mathbb{P}$ are no longer $2$-convex, thus far from being $\mathrm{CAT}(0)$. As a matter of fact, the natural metric on $\mathbb{P}$ turns out to be a modification of $d(\cdot,\cdot)$ in the form
\begin{equation}\label{eq:ThompsonMetric}
d_\infty(A,B):=\|\log(A^{-1/2}BA^{-1/2})\|=\mathrm{spr}\{\log(A^{-1}B)\}
\end{equation}
called the Thompson metric, which turns $(\mathbb{P},d_\infty)$ into a complete metric space such that the topology generated by $d_\infty$ agrees with the relative operator norm topology \cite{thompson}, where $\mathrm{spr}(X)$ denotes the spectral radius of $X$. This is also apparent from the inequality
\begin{equation}\label{eq:order}
e^{-d_\infty(X,Y)}\leq X^{-1/2}YX^{-1/2}\leq e^{d_\infty(X,Y)}
\end{equation}
for any $X,Y\in\mathbb{P}$ which is symmetric in $X$ and $Y$ \cite{palfia2}. Form the definition of $d_\infty$ through the spectral radius, one also derives
\begin{equation}\label{eq:invariance}
d_\infty(GXG^*,GYG^*)=d_\infty(X,Y)=d_\infty(X^{-1},Y^{-1})
\end{equation}
for any invertible $G\in\mathcal{B}(\mathcal{H})$. However there is a price to pay for the non-Hilbertian nature of $(\mathbb{S},\|\cdot\|)$ in this general case. The operator norm in \eqref{eq:ThompsonMetric} is an $L^{\infty}$-type, non-uniformly convex, non-smooth, non-differentiable norm, so in general $d_\infty$ admits infinitely many geodesics and thus $d_\infty^2$ is not uniformly convex either. Thus the metric formulation \eqref{eq:KarcherMean} of $\Lambda$ and all the corresponding metric theory of gradient flows \cite{AGS,bacak,ohta0,ohta,yokota} and proofs spectacularly break down, without such uniform convexity of the metric. In particular, even in a uniformly convex Banach space that is not Hilbert, the metric behavior of gradient curves of a convex function is not well understood at all \cite{AGS}. Thus Lawson-Lim \cite{lawsonlim1,lawsonlim12} building on \cite{limpalfia} adopted the critical point equation of the minimization problem in \eqref{eq:KarcherMean} to define $\Lambda$ in the general case as the unique solution $X\in\mathbb{P}$ of the \emph{Karcher equation}
\begin{equation*}
\int_{\mathbb{P}}\log_XAd\mu(A)=0
\end{equation*}
for $\mu=\sum_{i=1}^k\frac{1}{k}\delta_{A_i}$ with $A_i\in\mathbb{P}$ and $\log_XA:=X^{1/2}\log(X^{-1/2}AX^{-1/2})X^{1/2}$. Then monotonicity of $\Lambda$ in $\leq$ follows through an approximation of $\Lambda$ by the monotone family of power means \cite{limpalfia,lawsonlim1}. Then in \cite{limpalfia3} it was proved that this equation has a unique solution for fully supported Borel probability measures that integrate $d_\infty(\cdot,Y)$ and it preserves an appropriate generalization of the partial order $\leq$, called the \emph{stochastic order} of probability measures studied in detail in \cite{Hiai,lawson}.

The conjecture naturally arises that Sturm's strong law of large numbers for fully supported Borel probability measures that integrate $d_\infty$ should hold for the stochastic sequence $\{S_n\}_{n\in\mathbb{N}}$ in \eqref{eq:InductiveMeans} and its deterministic "nodice" counterpart by Holbrook \cite{Hol} in the $d_\infty$ topology. In this paper we prove both conjectures. The idea is to study the initial value problem
\begin{equation}\label{eq:IVP}
\dot{\gamma}(t)=\int_{\mathbb{P}}\log_{\gamma(t)}Ad\mu(A)
\end{equation}
for $t>0$, $\gamma(0)\in\mathbb{P}$ so that it recovers some of the Riemannian gradient flow structure of the minimization problem \eqref{eq:KarcherMean}. The idea to study such evolution problem related to $\Lambda$ first appeared recently in the work of the authors \cite{limpalfia3}, where it was used to confirm the conjecture in \cite{lawsonlim1} on the norm continuity of the family power means. In particular an important feature of a family of solution curves $\gamma(t):=S(t)\gamma(0)$ of \eqref{eq:IVP} is the exponential contraction property
\begin{equation*}
d_\infty\left(S(t)X,S(t)Y\right)\leq e^{-t}d_\infty(X,Y)
\end{equation*}
that follows from the Wasserstein contraction formula
\begin{equation*}
d_\infty(\Lambda(\mu),\Lambda(\nu))\leq W_1(\mu,\nu).
\end{equation*}

In section 2 we establish the necessary technical preliminaries, in section 3 we extend the domain of $\Lambda$ to probability measures with unbounded support. In section 4 we use the resolvent operator
\begin{equation*}
J_{\lambda}^{\mu}(X):=\Lambda\left(\frac{\lambda}{\lambda+1}\mu+\frac{1}{\lambda+1}\delta_X\right)\quad\text{for }\lambda>0
\end{equation*}
of \cite{limpalfia3} and prove a flexible metric inequality in Theorem~\ref{T:KobayashiEst} comparing distances of iterates of this map. This result itself can be viewed as a nonlinear generalizaton of an inequality for resolvent iterates in Banach spaces \cite{crandall} that have been successfully applied in the optimization society, see for instance in \cite{liu}. It also generalizes the estimates of \cite{limpalfia3}. As a consequence, we derive that the limit curve $S(t)X:=\lim_{n\to\infty}(J^\mu_{t/n})^n(X)$ exists and solves \eqref{eq:IVP}, and $S(t)\Lambda(\mu)=\Lambda(\mu)$ for $t\geq 0$. In section 5 we establish the necessary norm and $d_\infty$ estimates that then lead to the "nodice" Theorem~\ref{T:Nodice} proving the $d_\infty$ convergence of deterministic iterates $S_n$ of Holbrook in \eqref{eq:InductiveMeans} to $\Lambda(\sum_{i=0}^{k-1}\frac{1}{k}\delta_{A_i})$. Then building on this in section 6 we prove the almost sure convergence of $S_n$ to $\Lambda(\mu)$ in $d_\infty$ that confirms Sturm's strong law of large numbers in the most general $L^1$-case. In both cases our estimates reveal that for large enough $n$, $S_n$ provides a discrete time trajectory that stays close to some ODE curve $S(t_n)X_0$ which itself converges to $\Lambda(\mu)$ as $t_n\to\infty$.

We note that the $L^1$-integrability assumption on $d_\infty$ for $\mu$ is necessary for the almost sure convergence to hold in the strong law, since when $\dim{\mathcal{H}}=1$, the matrices in $\mathbb{P}$ commute and we have that the $\log:(0,\infty)\mapsto\mathbb{R}$ provides an isometry between the real line and $\mathbb{P}=(0,\infty)$, and there exists counterexamples to the almost sure convergence in the classical strong law in $\mathbb{R}$ when the expectation is not finite, see for example \cite{dudley,yokota}.

\section{Technical preliminaries}
In this paper we use the notation $X(\lambda)=O(\lambda)$ for $X(\lambda)\in\mathbb{S}$ in the sense that there exist constants $M,\lambda_0>0$, such that we have $\|X(\lambda)\|\leq M\lambda$ for all $0<\lambda\leq\lambda_0$ where $M$ does not depend on $\lambda$. This will be frequently used together with the Banach space version of Taylor's theorem and the mean value inequality for analytic functions.

The following result can be found in multiple resources, for a proof see for example \cite{lawson}.
\begin{proposition}\label{P:separableSupp}
Let $\mu$ be $\sigma$-additive Borel probability measures on $(\mathbb{P},d_\infty)$. Then the support $\supp(\mu)$ is separable. Moreover $\mu$ is fully supported, that is $\mu(\supp(\mu))=1$, if and only if $\mu$ is $\tau$-additive.
\end{proposition}

Let $\mathcal{P}^1(\mathbb{P})$ denote the convex set of $\tau$-additive Borel probability measures $\mu$ on $(\mathbb{P},\mathcal{B}(\mathbb{P}))$ such that $\int_{\mathbb{P}}d_\infty(X,A)d\mu(A)<+\infty$ for all $X\in\mathbb{P}$. We say that a sequence $\mu_n\in\mathcal{P}^1(\mathbb{P})$ is uniformly integrable if 
\begin{equation*}
\lim_{R\to \infty}\limsup_{n\to\infty}\int_{d_\infty(x,A)\geq R}d_\infty(x,A)d\mu_n(A)=0
\end{equation*}
for a (thus all) $x\in\mathbb{P}$. The $L^1$-Wasserstein distance between $\mu,\nu\in\mathcal{P}^1(\mathbb{P})$ is defined as
\begin{equation*}
W_1(\mu,\nu)=\inf_{\gamma\in\Pi(\mu,\nu)}\int_{\mathbb{P}\times\mathbb{P}}d_\infty(A,B)d\gamma(A,B)
\end{equation*}
where $\Pi(\mu,\nu)$ denotes the set of all $\tau$-additive Borel probability measures on the product space $\mathbb{P}\times\mathbb{P}$ with marginals $\mu$ and $\nu$. We consider only $\tau$-additive measures, since the following is not true in general for $\sigma$-additive Borel probability measures which are not fully supported, however its proof goes through for $\tau$-additive, equivalently fully supported, probability measures on compete metric spaces.

\begin{proposition}[Corollary 6.13. \cite{villani} \& Example 8.1.6., Theorem 8.10.45. \cite{bogachev}]\label{P:weakW1agree}
The topology generated by the Wasserstein metric $W_1(\cdot,\cdot)$ on $\mathcal{P}^1(\mathbb{P})$ agrees with the weak-$*$ (also called weak) topology of $\mathcal{P}^1(\mathbb{P})$ on uniformly integrable sequences of probability measures in $\mathcal{P}^1(\mathbb{P})$. Moreover finitely supported probability measures are $W_1$-dense in $\mathcal{P}^1(\mathbb{P})$.
\end{proposition}
\begin{proof}
According to Theorem 6.9 and Corollary 6.13 \cite{villani}, the assertion holds on a Polish metric space $(X,d)$. In our setting $(\mathbb{P},d_\infty)$ is a non-separable metric space, so we cannot directly apply these results. 
However given a sequence of probability measures $\mu_k\in\mathcal{P}^1(\mathbb{P})$ we claim that the first part of the assertion still holds. Indeed, by Proposition~\ref{P:separableSupp} we have that $\supp(\mu_k)$ is separable and by $\tau$-additivity it has full measure $\mu_k(\supp(\mu_k))=1$, thus we may take a countable dense subset $D_k\subseteq\supp(\mu_k)$ for each $k\in\mathbb{N}$. Then the union $D:=\cup_{k\in\mathbb{N}}D_k$ is still countable and dense in $\overline{\cup_{k\in\mathbb{N}}\supp(\mu_k)}$. Consider the Polish metric space $(\overline{D},d_\infty)$. On this space, the restriction of $W_1$ and the weak-$*$ topology of $\mathcal{P}^1(\mathbb{P})$ with uniform integrability coincide by Theorem 6.9 and Corollary 6.13 \cite{villani}. Moreover $\supp(\mu_k)\subseteq \overline{D}$, thus the assertion on the equivalence of topologies follows for the sequence $\mu_k$.

Then by Varadarajan's theorem which can be found as Theorem 11.4.1. in \cite{dudley} we have that for any $\mu\in\mathcal{P}^1(\mathbb{P})$ the empirical probability measures $\mu_n:=\sum_{i=1}^n\frac{1}{n}\delta_{Y_i}$ converge weakly to $\mu$ almost surely on the Polish metric space $(\supp(\mu),d_\infty)$, where $Y_i$ is a sequence of i.i.d. random variables on the Polish metric space $(\supp(\mu),d_\infty)$ with law $\mu$. 
So for each bounded continuous function $f$ on $(\supp(\mu),d_\infty)$ we have $\int_{\supp(\mu)}fd\mu_n\to\int_{\supp(\mu)}fd\mu$ which happens outside of a set of measure $0$. 
So on the complement we have weak convergence of $\mu_n$ to $\mu$. Now, one is left with checking that $\mu_n$ is a uniformly integrable sequence which follows from the uniform integrability of $\mu$ itself.
\end{proof}

\begin{definition}[strong measurability, Bochner integral]\label{D:BochnerIntegrable}
Let $(\Omega,\Sigma,\mu)$ be finite measure space and let $f:\Omega\mapsto\mathbb{P}$. Then $f$ is \emph{strongly measurable} if there exists a sequence of simple functions $f_n$, such that $\lim_{n\to\infty}f_n(\omega)=f(\omega)$ in the operator norm almost everywhere.

The function $f:\Omega\mapsto\mathbb{P}$ is \emph{Bochner integrable} if the following are satisfied:
\begin{itemize}
\item[(1)] $f$ is strongly measurable;
\item[(2)] there exists a sequence of simple functions $f_n$, such that $\lim_{n\to\infty}\int_{\Omega}\|f(\omega)-f_n(\omega)\|d\mu(\omega)=0$
\end{itemize}
In this case we define the \emph{Bochner integral} of $f$ by
$$\int_{\Omega}f(\omega)d\mu(\omega):=\lim_{n\to\infty}\int_{\Omega}f_n(\omega)d\mu(\omega).$$
\end{definition}
It is well known that a strongly measurable function $f$ on a finite measure space $(\Omega,\Sigma,\mu)$ is Bochner integrable if and only if $\int_{\Omega}\|f(\omega)\|d\mu(\omega)<\infty$.

For $X,A\in\mathbb{P}$ we use the notation
\begin{equation*}
\log_XA:=X^{1/2}\log(X^{-1/2}AX^{-1/2})X^{1/2}.
\end{equation*}
Notice that also $\log_XA=X\log(X^{-1}A)$ and the exponential metric
increasing (EMI) property (\cite{LL07})
\begin{eqnarray}\label{eq:EMI}||\log X-\log Y||\leq d_{\infty}(X,Y), \ \ \
X,Y\in {\Bbb P}. \end{eqnarray}

\begin{lemma}\label{L:LogIntegrable}
For all $\mu\in\mathcal{P}^1(\mathbb{P})$ and $X\in\mathbb{P}$, the Bochner integral $\int_{\mathbb{P}}\log_{X}Ad\mu(A)$ exists.
\end{lemma}
\begin{proof}
First of all, notice that $A\mapsto X\log(X^{-1}A)$ is strongly measurable, since $A\mapsto X\log(X^{-1}A)$ is norm continuous, hence $d_\infty$ continuous. Then
\begin{equation*}
\begin{split}
\int_{\mathbb{P}}\|X\log(X^{-1}A)\|d\mu(A)&\leq \int_{\mathbb{P}}\|X^{1/2}\|\|\log(X^{-1/2}AX^{-1/2})\|\|X^{1/2}\|d\mu(A)\\
&=\|X\|\int_{\mathbb{P}}\|\log(X^{-1/2}AX^{-1/2})\|d\mu(A)\\
&=\|X\|\int_{\mathbb{P}}d_\infty(X,A)d\mu(A)<\infty
\end{split}
\end{equation*}
which shows Bochner integrability.
\end{proof}

\begin{definition}[Karcher equation/mean]\label{D:Karcher}
For a $\mu\in\mathcal{P}^1(\mathbb{P})$ the \emph{Karcher equation} is defined as
\begin{equation}\label{eq:D:Karcher}
\int_{\mathbb{P}}\log_XAd\mu(A)=0,
\end{equation}
where $X\in\mathbb{P}$. If \eqref{eq:D:Karcher} has a unique solution in $X\in\mathbb{P}$, then it is called the Karcher mean and is denoted by $\Lambda(\mu)$.
\end{definition}

\begin{definition}[Weighted geometric mean]\label{D:GeometricMean}
Let $A,B\in\mathbb{P}$ and $t\in[0,1]$. Then for $(1-t)\delta_A+t\delta_B=:\mu\in\mathcal{P}^1(\mathbb{P})$ the Karcher equation
\begin{equation*}
\int_{\mathbb{P}}\log_XAd\mu(A)=(1-t)\log_XA+t\log_XB=0
\end{equation*}
has a unique solution $A\#_tB=\Lambda(\mu)$ called the \emph{weighted geometric mean} and
\begin{equation*}
A\#_tB=A^{1/2}\left(A^{-1/2}BA^{-1/2}\right)^tA^{1/2}=A\left(A^{-1}B\right)^t.
\end{equation*}
\end{definition}

By the dominated convergence theorem and Lemma~\ref{L:LogIntegrable} we have the following:
\begin{lemma}\label{L:gradContinuous}
For each $X\in\mathbb{P}$ and $\mu\in\mathcal{P}^1(\mathbb{P})$ the function $X\mapsto \int_{\mathbb{P}}\log_XAd\mu(A)$ is $d_\infty$ to norm continuous.
\end{lemma}
\begin{proof}
Pick a sequence $X_n\to X$ in the $d_\infty$ topology in $\mathbb{P}$. Then
\begin{equation}\label{eq1:L:gradContinuous}
\begin{split}
&\left\|\int_{\mathbb{P}}\log_{X_n}Ad\mu(A)-\int_{\mathbb{P}}\log_{X}Ad\mu(A)\right\|\\
&\leq\int_{\mathbb{P}}\left\|\log_{X_n}A-\log_{X}A\right\|d\mu(A)\\
&\leq\int_{\mathbb{P}}\left\|\log_{X_n}A\right\|+\left\|\log_{X}A\right\|d\mu(A)\\
&\leq\|X_n\|\int_{\mathbb{P}}d_\infty(X_n,A)d\mu(A)+\|X\|\int_{\mathbb{P}}d_\infty(X,A)d\mu(A)<\infty,
\end{split}
\end{equation}
thus $\left\|\log_{X_n}A-\log_{X}A\right\|$ is integrable. Since $d_\infty$ agrees with the relative norm topology, we have that
\begin{equation*}
F_n(A):=\left\|\log_{X_n}A-\log_{X}A\right\|\to 0
\end{equation*}
point-wisely for every $A\in\mathbb{P}$ as $n\to\infty$.
Then by the dominated convergence theorem we obtain
\begin{equation*}
\begin{split}
\lim_{n\to\infty}\int_{\mathbb{P}}\left\|\log_{X_n}A-\log_{X}A\right\|d\mu(A)&=\int_{\mathbb{P}}\lim_{n\to\infty}\left\|\log_{X_n}A-\log_{X}A\right\|d\mu(A)\\
&=0.
\end{split}
\end{equation*}
In view of \eqref{eq1:L:gradContinuous} this proves the assertion.
\end{proof}

For some further known facts below, see for example \cite{lawsonlim1}.

\begin{theorem}[see Theorem 6.4. \cite{lawsonlim1}]\label{T:KarcherExist}
Let $A_i\in\mathbb{P}$ for $1\leq i\leq n$ and let $\omega=(w_1,\ldots,w_n)$ be a probability vector. Then for $\mu=\sum_{i=1}^nw_i\delta_{A_i}$ the equation \eqref{eq:D:Karcher} has a unique positive definite solution $\Lambda(\mu)$.

In the special case $n=2$, we have
\begin{equation}\label{eq:T:KarcherExist}
\Lambda((1-t)\delta_{A}+t\delta_{B})=A\#_tB
\end{equation}
for any $t\in[0,1]$, $A,B\in\mathbb{P}$.
\end{theorem}

\begin{proposition}[see Proposition 2.5. \cite{lawsonlim1}]\label{P:KarcherW1contracts}
Let $A_i,B_i\in\mathbb{P}$ for $1\leq i\leq n$. Then $\Lambda$ for $\mu=\frac{1}{n}\sum_{i=1}^n\delta_{A_i}$ and $\nu=\frac{1}{n}\sum_{i=1}^n\delta_{B_i}$ satisfies
\begin{equation}\label{eq0:P:KarcherW1contracts}
d_{\infty}(\Lambda(\mu),\Lambda(\nu))\leq\sum_{i=1}^n\frac{1}{n}d_\infty(A_i,B_i),
\end{equation}
in particular by permutation invariance of $\Lambda$ in the variables $(A_1,\ldots,A_n)$ we have
\begin{equation}\label{eq:P:KarcherW1contracts}
d_{\infty}(\Lambda(\mu),\Lambda(\nu))\leq\min_{\sigma\in S_n}\sum_{i=1}^n\frac{1}{n}d_\infty(A_i,B_{\sigma(i)})=W_1(\mu,\nu).
\end{equation}
\end{proposition}


\section{Extension of $\Lambda$ to $\mathcal{P}^1(\mathbb{P})$}
We extend $\Lambda$ and its contraction properties by using continuity and contraction property of it with respect to $W_1$, along with the approximation properties of $\mathcal{P}^1(\mathbb{P})$ with respect to the metric $W_1$. The same technique was adopted in \cite{limpalfia3}.

\begin{lemma}\label{L:converge}
Let $x,y\in\mathbb{P}$ and $\mu_n,\mu\in\mathcal{P}^1(\mathbb{P})$ and that $\supp(\mu_n),\supp(\mu)\subseteq Z\subset\mathbb{P}$ where $Z$ is closed and separable. Assume also that $x\to y$ in $d_\infty$, $\mu_n\to \mu$ in $W_1$. Then
\begin{equation*}
\int_{\mathbb{P}}\log_{x}Ad\mu_n(A)\to \int_{\mathbb{P}}\log_{y}Ad\mu(A)
\end{equation*}
in the norm topology.
\end{lemma}
\begin{proof}
Let $x,y\in\mathbb{P}$ and $\mu,\nu\in\mathcal{P}^1(\mathbb{P})$. Then we have
\begin{equation}\label{eq1:L:converge}
\begin{split}
&\left\|\int_{\mathbb{P}}\log_{x}Ad\mu_n(A)-\int_{\mathbb{P}}\log_{y}Ad\mu(A) \right\| \\
&\leq \left\|\int_{\mathbb{P}}\log_{x}Ad\mu(A)-\int_{\mathbb{P}}\log_{y}Ad\mu(A) \right\|\\
&\quad +\left\|\int_{\mathbb{P}}\log_{x}Ad\mu_n(A)-\int_{\mathbb{P}}\log_{x}Ad\mu(A) \right\|\\
&\leq \left\|\int_{\mathbb{P}}\log_{x}Ad\mu(A)-\int_{\mathbb{P}}\log_{y}Ad\mu(A) \right\|\\
&\quad +\|x\|\int_{\mathbb{P}}\|\log(x^{-1/2}Ax^{-1/2})\|d(\mu_n-\mu)(A)\\
&\leq \left\|\int_{\mathbb{P}}\log_{x}Ad\mu(A)-\int_{\mathbb{P}}\log_{y}Ad\mu(A) \right\|\\
&\quad +\|x\|\int_{\mathbb{P}}d_\infty(x,A)d(\mu_n-\mu)(A).
\end{split}
\end{equation}
If $x\to y$ in $d_\infty$, then the first term in the above converges to $0$ by Lemma~\ref{L:gradContinuous}. The second term goes to $0$ by Proposition~\ref{P:weakW1agree}.
\end{proof}

\begin{theorem}\label{T:LambdaExists}
For all $\mu\in\mathcal{P}^1(\mathbb{P})$ there exists a solution of \eqref{eq:D:Karcher} denoted by $\Lambda(\mu)$ (with an obvious abuse of notation), which satisfies
\begin{equation}\label{eq:T:LambdaExists}
d_\infty(\Lambda(\mu),\Lambda(\nu))\leq W_1(\mu,\nu)
\end{equation}
for all $\nu\in\mathcal{P}^1(\mathbb{P})$.
\end{theorem}
\begin{proof}
Let $\mu\in\mathcal{P}^1(\mathbb{P})$. Then by Proposition~\ref{P:weakW1agree} there exists a $W_1$-convergent sequence of finitely supported probability measures $\mu_n\in\mathcal{P}^1(\mathbb{P})$ such that $W_1(\mu,\mu_n)\to 0$. By Theorem~\ref{T:KarcherExist} $\Lambda(\mu_n)$ exists for any $n$ in the index set. We also have that $W_1(\mu_m,\mu_n)\to 0$ as $m,n\to\infty$ and by \eqref{eq:P:KarcherW1contracts} it follows that $d_\infty(\Lambda(\mu_m),\Lambda(\mu_n))\to 0$ as $m,n\to\infty$, i.e. $\Lambda(\mu_n)$ is a $d_\infty$ Cauchy sequence. Thus we define
\begin{equation*}
\tilde{\Lambda}(\mu):=\lim_{n\to\infty}\Lambda(\mu_n).
\end{equation*}
Since \eqref{eq:T:LambdaExists} holds by Proposition~\ref{P:KarcherW1contracts} for finitely supported probability measures, we extend \eqref{eq:T:LambdaExists} to the whole of $\mathcal{P}^1(\mathbb{P})$ by $W_1$-continuity, using the $W_1$-density of finitely supported probability measures in $\mathcal{P}^1(\mathbb{P})$.

Then by construction for all $n$ we have 
\begin{equation*}
\int_{\mathbb{P}}\log_{\Lambda(\mu_n)}Ad\mu_n(A)=0,
\end{equation*}
thus by Lemma~\ref{L:converge} we have 
\begin{equation*}
\int_{\mathbb{P}}\log_{\Lambda(\mu_n)}Ad\mu_n(A)\to \int_{\mathbb{P}}\log_{\tilde{\Lambda}(\mu)}Ad\mu(A),
\end{equation*}
that is
\begin{equation*}
\int_{\mathbb{P}}\log_{\tilde{\Lambda}(\mu)}Ad\mu(A)=0.
\end{equation*}
\end{proof}

\begin{definition}[Karcher mean]\label{D:KarcherMean}
Given $\mu\in\mathcal{P}^1(\mathbb{P})$ with unbounded support we define $\Lambda(\mu)$ as the limit obtained in Theorem~\ref{T:LambdaExists}.
\end{definition}

\section{Evolution systems related to $\Lambda$}
The fundamental $W_1$-contraction property \eqref{eq:T:LambdaExists} enables us to develop an ODE flow theory for $\Lambda$ that resembles the gradient flow theory for its potential function in the finite dimensional $\mathrm{CAT}(0)$-space case, see \cite{limpalfia2,ohta} and the monograph \cite{bacak}. Given a $\mathrm{CAT}(\kappa)$-space $(X,d)$, the \emph{Moreau-Yosida} resolvent of a lower semi-continuous function $f$ is defined as
\begin{equation*}
J_{\lambda}(x)=\argmin_{y\in X}f(y)+\frac{1}{2\lambda}d^2(x,y)
\end{equation*}
for $\lambda>0$. Then the gradient flow semigroup of $f$ is defined as
\begin{equation*}
S(t)x_0=\lim_{n\to\infty}(J_{t/n})^nx_0
\end{equation*}
for $t\in[0,\infty)$ and starting point $x_0\in X$, see \cite{bacak}. However in the infinite dimensional case substituting $d_\infty$ in place of $d$ in the above formulas leads to many difficulties. Furthermore the potential function $f$ is not known to exist in the infinite dimensional case of $\mathbb{P}$. However if we use the formulation of the critical point gradient equation equivalent to the definition of $J_\lambda$ above, we can obtain a reasonable ODE theory in our setting for $\Lambda$.

\begin{definition}[Resolvent operator]
Given $\mu\in\mathcal{P}^1(\mathbb{P})$ we define the resolvent operator for $\lambda>0$ and $X\in\mathbb{P}$ as
\begin{equation}\label{eq:D:resolvent}
J_{\lambda}^{\mu}(X):=\Lambda\left(\frac{\lambda}{\lambda+1}\mu+\frac{1}{\lambda+1}\delta_X\right),
\end{equation}
a solution we obtained in Theorem~\ref{T:LambdaExists} of the Karcher equation
\begin{equation*}
\frac{\lambda}{\lambda+1}\int_{\mathbb{P}}\log_{Z}Ad\mu(A)+\frac{1}{\lambda+1}\log_{Z}(X)=0
\end{equation*}
for $Z\in\mathbb{P}$ according to Definition~\ref{D:KarcherMean}.
\end{definition}

We readily obtain the following fundamental contraction property of the resolvent.

\begin{proposition}[Resolvent contraction]\label{P:ResolventContraction}
Given $\mu\in\mathcal{P}^1(\mathbb{P})$, for $\lambda>0$ and $X,Y\in\mathbb{P}$ we have
\begin{equation}\label{eq:P:ResolventContraction}
d_\infty(J_{\lambda}^\mu(X),J_{\lambda}^\mu(Y))\leq \frac{1}{1+\lambda}d_\infty(X,Y).
\end{equation}
\end{proposition}
\begin{proof}
Let $\mu_i\in\mathcal{P}^1(\mathbb{P})$ be a sequence of finitely supported probability measures $W_1$-converging to $\mu$ by Proposition~\ref{P:weakW1agree}. Then by the triangle inequality and Proposition~\ref{P:KarcherW1contracts} we get
\begin{equation*}
\begin{split}
&d_\infty(J_{\lambda}^\mu(X),J_{\lambda}^\mu(Y))\\
&\leq d_\infty(J_{\lambda}^\mu(X),J_{\lambda}^{\mu_i}(X))+d_\infty(J_{\lambda}^{\mu_i}(X),J_{\lambda}^{\mu_i}(Y))+d_\infty(J_{\lambda}^{\mu_i}(Y),J_{\lambda}^\mu(Y))\\
&\leq d_\infty(J_{\lambda}^\mu(X),J_{\lambda}^{\mu_i}(X))+\frac{1}{1+\lambda}d_\infty(X,Y)+d_\infty(J_{\lambda}^{\mu_i}(Y),J_{\lambda}^\mu(Y)).
\end{split}
\end{equation*}
Since $d_\infty(J_{\lambda}^\mu(Z),J_{\lambda}^{\mu_i}(Z))\to 0$ as $i\to\infty$ by \eqref{eq:T:LambdaExists}, taking the limit $i\to\infty$ in the above chain of inequalities yields the assertion.
\end{proof}

\begin{proposition}[Resolvent identity]\label{P:ResolventIdentity}
Given $\mu\in\mathcal{P}^1(\mathbb{P})$, for $\tau>\lambda>0$ and $X\in\mathbb{P}$ we have
\begin{equation}\label{eq:P:ResolventIdentity}
J_{\tau}^\mu(X)=J_{\lambda}^\mu\left(J_{\tau}^\mu(X)\#_{\frac{\lambda}{\tau}}X\right).
\end{equation}
\end{proposition}
\begin{proof}
First suppose that $\mu=\sum_{i=1}^nw_i\delta_{A_i}$ where $A_i\in\mathbb{P}$ for $1\leq i\leq n$ and $\omega=(w_1,\ldots,w_n)$ a probability vector. By \eqref{eq:D:resolvent} we have
\begin{equation*}
\tau\int_{\mathbb{P}}\log_{J_{\tau}^\mu(X)}Ad\mu(A)+\log_{J_{\tau}^\mu(X)}X=0
\end{equation*}
and from that it follows that
\begin{equation*}
\begin{split}
\lambda\int_{\mathbb{P}}\log_{J_{\tau}^\mu(X)}Ad\mu(A)+\frac{\lambda}{\tau}\log_{J_{\tau}^\mu(X)}X&=0,\\
\lambda\int_{\mathbb{P}}\log_{J_{\tau}^\mu(X)}Ad\mu(A)+\log_{J_{\tau}^\mu(X)}\left(J_{\tau}^\mu(X)\#_{\frac{\lambda}{\tau}}X\right)&=0,
\end{split}
\end{equation*}
and the above equation still uniquely determines $J_{\tau}^\mu(X)$ as its only positive solution by Theorem~\ref{T:KarcherExist}, thus establishing \eqref{eq:P:ResolventIdentity} for finitely supported measures $\mu$.

The general $\mu\in\mathcal{P}^1(\mathbb{P})$ case of \eqref{eq:P:ResolventIdentity} is obtained by approximating $\mu$ in $W_1$ by a sequence of finitely supported measures $\mu_{i}\in\mathcal{P}^1(\mathbb{P})$ and using \eqref{eq:T:LambdaExists} to show that $J_{\lambda}^{\mu_i}(X)\to J_{\lambda}^\mu(X)$ in $d_\infty$ and also the fact that $\#_t$ appearing in \eqref{eq:P:ResolventIdentity} is also $d_\infty$-continuous, hence obtaining \eqref{eq:P:ResolventIdentity} in the limit as $\mu_i\to\mu$ in $W_1$.
\end{proof}

\begin{proposition}\label{P:ResolventBound}
Given $\mu\in\mathcal{P}^1(\mathbb{P})$, $\lambda>0$ and $X\in\mathbb{P}$ we have
\begin{equation}\label{eq:P:ResolventBound}
\begin{split}
d_\infty(J_{\lambda}^\mu(X),X)&\leq \frac{\lambda}{1+\lambda}\int_{\mathbb{P}}d_\infty(X,A)d\mu(A)\\
d_\infty\left(J_{\lambda_1}^\mu\circ\cdots\circ J_{\lambda_n}^\mu(X),X\right)&\leq \sum_{i=1}^{n}\frac{\lambda_i}{1+\lambda_i}\int_{\mathbb{P}}d_\infty(X,A)d\mu(A).
\end{split}
\end{equation}
\end{proposition}
\begin{proof}
By Theorem~\ref{T:LambdaExists} $J_{\lambda}^\mu(X)$ is a solution of
\begin{equation}\label{eq1:P:ResolventBound}
\lambda\int_{\mathbb{P}}\log_{J_{\lambda}^\mu(X)}Ad\mu(A)+\log_{J_{\lambda}^\mu(X)}X=0,
\end{equation}
hence we have
\begin{equation*}
\begin{split}
d_\infty(J_{\lambda}^\mu(X),X)&=\left\|\log\left(J_{\lambda}^\mu(X)^{-1/2}XJ_{\lambda}^\mu(X)^{-1/2}\right)\right\|\\
&=\lambda\left\|\int_{\mathbb{P}}\log\left(J_{\lambda}^\mu(X)^{-1/2}AJ_{\lambda}^\mu(X)^{-1/2}\right)d\mu(A)\right\|\\
&\leq \lambda\int_{\mathbb{P}}\left\|\log\left(J_{\lambda}^\mu(X)^{-1/2}AJ_{\lambda}^\mu(X)^{-1/2}\right)\right\|d\mu(A)\\
&=\lambda\int_{\mathbb{P}}d_\infty(J_{\lambda}^\mu(X),A)d\mu(A)
\end{split}
\end{equation*}
Given $J_{\lambda}^\mu(X)\in\mathbb{P}$ we can solve \eqref{eq1:P:ResolventBound} for $X\in\mathbb{P}$, thus by Proposition~\ref{P:ResolventContraction} we also have
\begin{equation*}
d_\infty(J_{\tau}^\mu(X),X)=d_\infty\left(J_{\tau}^\mu(X),J_{\tau}^\mu\left(\left(J_{\tau}^\mu\right)^{-1}(X)\right)\right)\leq \frac{1}{1+\lambda}d_\infty\left(X,\left(J_{\tau}^\mu\right)^{-1}(X)\right),
\end{equation*}
hence the first inequality in \eqref{eq:P:ResolventBound} follows.

The second inequality in \eqref{eq:P:ResolventBound} follows from the first by the estimate
\begin{equation*}
\begin{split}
d_\infty\left(J_{\lambda_1}^\mu\circ\cdots\circ J_{\lambda_n}^\mu(X),X\right)&\leq \sum_{i=0}^{n-1}d_\infty\left(J_{\lambda_1}^\mu\circ\cdots\circ J_{\lambda_{n-i}}^\mu(X),J_{\lambda_1}^\mu\circ\cdots\circ J_{\lambda_{n-i-1}}^\mu(X)\right)\\
&\leq \sum_{i=0}^{n-1}\prod_{j=1}^{n-i-1}(1+\lambda_j)^{-1}d_\infty\left(J_{\lambda_{n-i}}^\mu(X),X\right)\\
&\leq \sum_{i=1}^{n}d_\infty\left(J_{\lambda_{i}}^\mu(X),X\right)\\
&=\sum_{i=1}^{n}\frac{\lambda_i}{1+\lambda_i}\int_{\mathbb{P}}d_\infty(X,A)d\mu(A).
\end{split}
\end{equation*}
\end{proof}

The following estimate is a generalization of the one in \cite{limpalfia3} in the sense that it allows non-uniform subdivisions of the time interval $[0,\infty)$.

\begin{theorem}\label{T:KobayashiEst}
Let $\{t_i\}_{i\in\mathbb{N}},\{\hat{t}_j\}_{j\in\mathbb{N}}$ denote sequences such that $t_i,\hat{t}_j\in[0,\infty)$ and $t_{i+1}>t_i,\hat{t}_{j+1}>\hat{t}_j$. Let $X\in\mathbb{P}$ and $\mu\in\mathcal{P}^1(\mathbb{P})$ and $\tau_i:=t_{i}-t_{i-1}$, $\hat{\tau}_i:=\hat{t}_{i}-\hat{t}_{i-1}$. Let $X_{0}=\hat{X}_{0}=X$ and define $X_{i+1}:=J_{\tau_{i+1}}^\mu(X_{i})$ and $\hat{X}_{i+1}:=J_{\hat{\tau}_{i+1}}^\mu(\hat{X}_{i})$. Let $m,n\in\mathbb{N}$. 
Then
\begin{equation}\label{eq:T:KobayashiEst}
d_\infty(X_m,\hat{X}_n)\leq \left[\prod_{j=1}^{\min\{m,n\}}(1+\min\{\tau_j,\hat{\tau}_j\})^{-1}\right]\left[(t_m-\hat{t}_n)^2+\sigma_m+\hat{\sigma}_n\right]^{1/2}C
\end{equation}
where $\sigma_m=\sum_{i=1}^m(t_i-t_{i-1})^2,\hat{\sigma}_n=\sum_{i=1}^n(\hat{t}_i-\hat{t}_{i-1})^2$ and $C=\int_{\mathbb{P}}d_\infty(X,A)d\mu(A)$.
\end{theorem}
\begin{proof}
We prove \eqref{eq:T:KobayashiEst} by induction on $n,m\in\mathbb{N}$. Let $a_{m,n}:=d_\infty(X_m,\hat{X}_n)$. Firstly, it follows from \eqref{eq:P:ResolventBound} that $a_{0,k}$ satisfies \eqref{eq:T:KobayashiEst}, and by symmetry $a_{k,0}$ as well for $k\in\mathbb{N}$.

Assume \eqref{eq:T:KobayashiEst} holds for $n,m\in\mathbb{N}$. Assume first that $\hat{\tau}_{n+1}\geq \tau_{m+1}$. Then
\begin{equation*}
\begin{split}
a_{m+1,n+1}&=d_\infty\left(J^\mu_{\tau_{m+1}}(X_m),J^\mu_{\hat{\tau}_{n+1}}(\hat{X}_n)\right)\\
&=d_\infty\left(J_{\tau_{m+1}}^\mu(X_m),J_{\tau_{m+1}}^\mu\left(J_{\hat{\tau}_{n+1}}^\mu(\hat{X}_n)\#_{\frac{\tau_{m+1}}{\hat{\tau}_{n+1}}}\hat{X}_n\right)\right)\\
&\leq(1+\tau_{m+1})^{-1}d_\infty\left(X_m,J_{\hat{\tau}_{n+1}}^\mu(\hat{X}_n)\#_{\frac{\tau_{m+1}}{\hat{\tau}_{n+1}}}\hat{X}_n\right)\\
&\leq(1+\tau_{m+1})^{-1}\left[\frac{\tau_{m+1}}{\hat{\tau}_{n+1}}d_\infty(X_m,\hat{X}_n)+\frac{\hat{\tau}_{n+1}-\tau_{m+1}}{\hat{\tau}_{n+1}}d_\infty(X_m,J_{\hat{\tau}_{n+1}}^\mu(\hat{X}_n))\right]\\
&=(1+\tau_{m+1})^{-1}\left(\frac{\tau_{m+1}}{\hat{\tau}_{n+1}}a_{m,n}+\frac{\hat{\tau}_{n+1}-\tau_{m+1}}{\hat{\tau}_{n+1}}a_{m,n+1}\right)\\
&\leq(1+\tau_{m+1})^{-1}\left(\frac{\tau_{m+1}}{\hat{\tau}_{n+1}}+\frac{\hat{\tau}_{n+1}-\tau_{m+1}}{\hat{\tau}_{n+1}}\right)^{1/2}\\
&\quad\times\left(\frac{\tau_{m+1}}{\hat{\tau}_{n+1}}a^2_{m,n}+\frac{\hat{\tau}_{n+1}-\tau_{m+1}}{\hat{\tau}_{n+1}}a^2_{m,n+1}\right)^{1/2}\\
&=(1+\tau_{m+1})^{-1}\left(\frac{\tau_{m+1}}{\hat{\tau}_{n+1}}a^2_{m,n}+\frac{\hat{\tau}_{n+1}-\tau_{m+1}}{\hat{\tau}_{n+1}}a^2_{m,n+1}\right)^{1/2},
\end{split}
\end{equation*}
where first we used the Resolvent Identity \eqref{eq:P:ResolventIdentity}, followed by the Contraction Property \eqref{eq:P:ResolventContraction}, then the Convexity \eqref{eq0:P:KarcherW1contracts} and finally the Cauchy-Schwarz inequality. In particular it follows that
\begin{equation}\label{eq1:T:KobayashiEst}
\hat{\tau}_{n+1}(1+\tau_{m+1})^{2}a^2_{m+1,n+1}\leq\tau_{m+1}a^2_{m,n}+(\hat{\tau}_{n+1}-\tau_{m+1})a^2_{m,n+1},
\end{equation}
and by using the induction hypothesis
\begin{equation*}
\begin{split}
&\hat{\tau}_{n+1}(1+\tau_{m+1})^{2}a^2_{m+1,n+1}\\
&\leq\tau_{m+1}\prod_{j=1}^{\min\{m,n\}}(1+\min\{\tau_j,\hat{\tau}_j\})^{-2}\left[(t_m-\hat{t}_n)^2+\sigma_m+\hat{\sigma}_n\right]C^2\\
&+(\hat{\tau}_{n+1}-\tau_{m+1})\prod_{j=1}^{\min\{m,n+1\}}(1+\min\{\tau_j,\hat{\tau}_j\})^{-2}\left[(t_m-\hat{t}_{n+1})^2+\sigma_m+\hat{\sigma}_{n+1}\right]C^2.
\end{split}
\end{equation*}
Thus, what remains to be verified is that
\begin{equation*}
\begin{split}
&\tau_{m+1}\prod_{j=1}^{\min\{m,n\}}(1+\min\{\tau_j,\hat{\tau}_j\})^{-2}\left[(t_m-\hat{t}_n)^2+\sigma_m+\hat{\sigma}_n\right]\\
&+(\hat{\tau}_{n+1}-\tau_{m+1})\prod_{j=1}^{\min\{m,n+1\}}(1+\min\{\tau_j,\hat{\tau}_j\})^{-2}\left[(t_m-\hat{t}_{n+1})^2+\sigma_m+\hat{\sigma}_{n+1}\right]\\
&\leq\hat{\tau}_{n+1}(1+\tau_{m+1})^{2}\prod_{j=1}^{\min\{m+1,n+1\}}(1+\min\{\tau_j,\hat{\tau}_j\})^{-2}\\
&\quad\times\left[(t_{m+1}-\hat{t}_{n+1})^2+\sigma_{m+1}+\hat{\sigma}_{n+1}\right]
\end{split}
\end{equation*}
which follows if
\begin{equation*}
\begin{split}
&\tau_{m+1}\left[(t_m-\hat{t}_n)^2+\sigma_m+\hat{\sigma}_n\right]+\hat{\tau}_{n+1}\left[(t_m-\hat{t}_{n+1})^2+\sigma_m+\hat{\sigma}_{n+1}\right]\\
&\leq\hat{\tau}_{n+1}\left[(t_{m+1}-\hat{t}_{n+1})^2+\sigma_{m+1}+\hat{\sigma}_{n+1}\right]+\tau_{m+1}\left[(t_m-\hat{t}_{n+1})^2+\sigma_m+\hat{\sigma}_{n+1}\right].
\end{split}
\end{equation*}
After substitution, expanding the terms and cancellation the above simplifies to
\begin{equation*}
(t_{m+1}-t_{m})^2(\hat{t}_{n+1}-\hat{t}_n)\geq 0
\end{equation*}
which is trivially satisfied by the assumptions on the sequences $t_j$ and $\hat{t}_j$.

In the other case when $\hat{\tau}_{n+1}<\tau_{m+1}$, we can follow the same argument to obtain that $a_{m+1,n+1}$ satisfies \eqref{eq:T:KobayashiEst}. In the remaining cases when either $m=0$ or $n=0$ we arrive at \eqref{eq:T:KobayashiEst} by \eqref{eq:P:ResolventBound}. This completes the induction on $n,m\in\mathbb{N}$, and the proof is complete.
\end{proof}

\begin{theorem}\label{T:ExponentialFormula}
For any $X,Y\in\mathbb{P}$ and $t>0$ the curve
\begin{equation}\label{eq1:T:ExponentialFormula}
S(t)X:=\lim_{n\to\infty}\left(J_{t/n}^{\mu}\right)^{n}(X)
\end{equation}
exists where the limit is in the $d_\infty$-topology and it is Lipschitz-continuous on compact time intervals $[0,T]$ for any $T>0$. Moreover it satisfies the contraction property
\begin{equation}\label{eq2:T:ExponentialFormula}
d_\infty\left(S(t)X,S(t)Y\right)\leq e^{-t}d_\infty(X,Y),
\end{equation}
and for $s>0$ verifies the semigroup property
\begin{equation}\label{eq3:T:ExponentialFormula}
S(t+s)X=S(t)(S(s)X),
\end{equation}
and the flow operator $S:\mathbb{P}\times(0,\infty)\mapsto \mathbb{P}$ extends by $d_\infty$-continuity to $S:\mathbb{P}\times[0,\infty)\mapsto \mathbb{P}$.
\end{theorem}
\begin{proof}
The proof follows that of Theorem I in \cite{crandall} using the previous estimates of this section. In particular for $n\geq m>0$ one obtains
\begin{equation}\label{eq4:T:ExponentialFormula}
d_\infty\left(\left(J_{t/n}^\mu\right)^{n}(X),\left(J_{t/m}^\mu\right)^{m}(X)\right)\leq t\left(\frac{1}{m}+\frac{1}{n}\right)^{1/2}\int_{\mathbb{P}}d_\infty(X,A)d\mu(A),
\end{equation}
so $\lim_{n\to\infty}\left(J_{t/n}^\mu\right)^{n}(X)$ exists proving \eqref{eq1:T:ExponentialFormula}. Also by \eqref{eq:P:ResolventContraction}, $\left(J_{t/n}^\mu\right)^{n}$ satisfies
\begin{equation*}
d_\infty\left(\left(J_{t/n}^\mu\right)^{n}(X),\left(J_{t/n}^\mu\right)^{n}(Y)\right)\leq \left(1+\frac{t}{n}\right)^{-n}d_\infty(X,Y),
\end{equation*}
hence also \eqref{eq2:T:ExponentialFormula}. We also have
\begin{equation}\label{eq5:T:ExponentialFormula}
d_\infty\left(S(s)X,S(t)X\right)\leq |s-t|\int_{\mathbb{P}}d_\infty(X,A)d\mu(A)
\end{equation}
proving Lipschitz-continuity in $t$ on compact time intervals. The proof of the semigroup property is exactly the same as in \cite{crandall}.
\end{proof}

We need some basic estimates for the remainder of Taylor series expansions of $\exp(X)$ and $\log(X)$ for self-adjoint $X\in\mathbb{S}$.
\begin{lemma}\label{L:SeriesExpansions}
Let $X\in\mathbb{S}$. Then we have
\begin{equation}\label{eq:L:SeriesExpansions1}
\|\exp(X)-(X+I)\|\leq \|X\|^2\frac{e^{\|X\|}}{2}.
\end{equation}
Moreover for $\|X-I\|<1$ we have
\begin{equation}\label{eq:L:SeriesExpansions2}
\|\log(X)-(X-I)\|\leq \frac{\|X-I\|^2}{2(1-\|X-I\|)}.
\end{equation}
\end{lemma}
\begin{proof}
We have
\begin{equation*}
\begin{split}
\|\exp(X)-(X+I)\|&=\left\|\sum_{k=0}^\infty\frac{1}{k!}X^k-X-I\right\|\leq\sum_{k=2}^\infty\frac{\|X\|^k}{k!}\\
&\leq\|X\|^2\sum_{k=0}^\infty\frac{\|X\|^k}{(k+2)!}\leq \|X\|^2\frac{e^{\|X\|}}{2}
\end{split}
\end{equation*}
establishing \eqref{eq:L:SeriesExpansions1}. If $\|X-I\|<1$ we have
\begin{equation*}
\begin{split}
\|\log(X)-(X-I)\|&=\left\|\sum_{k=1}^\infty\frac{(-1)^{k-1}}{k}(X-I)^k-(X-I)\right\|\leq\sum_{k=2}^\infty\frac{\|X-I\|^k}{k}\\
&\leq\frac{\|X-I\|^2}{2}\sum_{k=0}^\infty\|X-I\|^k\leq \frac{\|X-I\|^2}{2(1-\|X-I\|)}
\end{split}
\end{equation*}
proving \eqref{eq:L:SeriesExpansions2}.
\end{proof}

Before stating the next result we need another auxiliary lemma describing the asymptotic behavior of $J_{t/n}^{\mu}(X)$.

\begin{lemma}\label{L:ResolventAsymptotics}
Let $\mu\in\mathcal{P}^1(\mathbb{P})$, $X\in\mathbb{P}$ and $C:=\int_{\mathbb{P}}d_\infty(X,A)d\mu(A)$. Then for any $\frac{\log(2)}{C}>\lambda>0$ we have
\begin{equation}\label{eq:L:ResolventAsymptotics}
\log_{J_{\lambda}^{\mu}(X)}X=X-J_{\lambda}^{\mu}(X)+J_{\lambda}^{\mu}(X)^{1/2}O\left((C\lambda)^2\right)J_{\lambda}^{\mu}(X)^{1/2}.
\end{equation}
\end{lemma}
\begin{proof}
By Proposition~\ref{P:ResolventBound}
\begin{equation*}
e^{-\lambda\left(1+\lambda\right)^{-1}C}-I\leq J_{\lambda}^{\mu}(X)^{-1/2}XJ_{\lambda}^{\mu}(X)^{-1/2}-I\leq e^{\lambda\left(1+\lambda\right)^{-1}C}-I,
\end{equation*}
hence
\begin{equation*}
e^{-\lambda C}-I\leq J_{\lambda}^{\mu}(X)^{-1/2}XJ_{\lambda}^{\mu}(X)^{-1/2}-I\leq e^{\lambda C}-I.
\end{equation*}
Now $(e^{\lambda C}-1)^2\geq 0$ and the assumption implies $e^{\lambda C}-1\geq 1-e^{-\lambda C}\geq 0$, thus the above yields
\begin{equation}\label{eq1:L:ResolventAsymptotics}
\|J_{\lambda}^{\mu}(X)^{-1/2}XJ_{\lambda}^{\mu}(X)^{-1/2}-I\|\leq e^{\lambda C}-1\leq O(\lambda C).
\end{equation}
Thus in view of the series expansion \eqref{eq:L:SeriesExpansions2}, we get
\begin{equation*}
\begin{split}
\log\left(J_{\lambda}^{\mu}(X)^{-1/2}XJ_{\lambda}^{\mu}(X)^{-1/2}\right)&\\
=J_{\lambda}^{\mu}(X)^{-1/2}&XJ_{\lambda}^{\mu}(X)^{-1/2}-I+O\left((C\lambda)^2\right),
\end{split}
\end{equation*}
from which the assertion follows.
\end{proof}

The proof of the following theorem, in essence, is analogous to that of Theorem II in \cite{crandall}.

\begin{theorem}\label{T:StrongSolution}
Let $\mu\in\mathcal{P}^1(\mathbb{P})$ and $X\in\mathbb{P}$. Then for $t>0$, the curve $X(t):=S(t)X$ provides a strong solution of the Cauchy problem
\begin{equation*}
\begin{split}
X(0)&:=X,\\
\dot{X}(t)&=\int_{\mathbb{P}}\log_{X(t)}Ad\mu(A),
\end{split}
\end{equation*}
where the derivative $\dot{X}(t)$ is the Fr\'echet-derivative.
\end{theorem}
\begin{proof}
Due to the semigroup property of $S(t)$, it is enough to check that
\begin{equation*}
\lim_{t\to 0+}\frac{S(t)X-X}{t}=\int_{\mathbb{P}}\log_{X}Ad\mu(A)
\end{equation*}
where the limit is in the norm topology. We have
\begin{equation*}
\begin{split}
\frac{S(t)X-X}{t}&=\lim_{n\to\infty}\frac{\left(J_{t/n}^{\mu}\right)^n(X)-X}{t}\\
&=\lim_{n\to\infty}\frac{1}{n}\frac{\sum_{i=0}^{n-1}J_{t/n}^{\mu}\left(\left(J_{t/n}^{\mu}\right)^i(X)\right)-\left(J_{t/n}^{\mu}\right)^i(X)}{t/n}
\end{split}
\end{equation*}
and also
\begin{equation}\label{eq1:T:StrongSolution}
\frac{t}{n}\int_{\mathbb{P}}\log_{\left(J_{t/n}^{\mu}\right)^i(X)}Ad\mu(A)+\log_{\left(J_{t/n}^{\mu}\right)^i(X)}\left(J_{t/n}^{\mu}\right)^{i-1}(X)=0.
\end{equation}
Then assuming that $t>0$ is small enough, the estimates in Proposition~\ref{P:ResolventBound} imply that $d_\infty\left(\left(J_{t/n}^{\mu}\right)^i(X),\left(J_{t/n}^{\mu}\right)^{i-1}(X)\right)$ is arbitrarily small. Thus we use Lemma~\ref{L:ResolventAsymptotics} to expand the second term in \eqref{eq1:T:StrongSolution} and then sum up the resulting equations for $0\leq i\leq n-1$ to obtain
\begin{equation*}
\begin{split}
\frac{S(t)X-X}{t}&=\lim_{n\to\infty}\frac{1}{n}\sum_{i=0}^{n-1}\int_{\mathbb{P}}\log_{\left(J_{t/n}^{\mu}\right)^i(X)}Ad\mu(A)\\
&\quad+\left(J_{t/n}^{\mu}\right)^i(X)^{1/2}O\left(\frac{t}{n}\right)\left(J_{t/n}^{\mu}\right)^i(X)^{1/2},
\end{split}
\end{equation*}
which combined with Lemma~\ref{L:gradContinuous} and the second estimate of \eqref{eq:P:ResolventBound} proves the assertion.
\end{proof}

The following result established the uniqueness of the solution of the Karcher equation \eqref{eq:D:Karcher} for measures with bounded support. It can be used to prove the uniqueness of the solution of \eqref{eq:D:Karcher} for general elements with unbounded support in $\mathcal{P}^1(\mathbb{P})$.

\begin{theorem}[Theorem 6.13. \& Example 6.1. in \cite{palfia2}]\label{T:KarcherExist2}
Let $\mu\in\mathcal{P}^1(\mathbb{P})$ such that $\supp(\mu)$ is bounded. Then the Karcher equation \eqref{eq:D:Karcher} has a unique positive definite solution $\Lambda(\mu)$.
\end{theorem}

The following convexity of the Wasserstein metric $W_1$ is well known, see for instance \cite{villani}. We provide its proof for completeness.

\begin{proposition}\label{P:W_1convex}
The $W_1$ distance is convex, that is for $\mu_1,\mu_2,\nu_1,\nu_2\in\mathcal{P}^1(\mathbb{P})$ and $t\in[0,1]$ we have
\begin{equation}\label{eq:P:W_1convex}
W_1((1-t)\mu_1+t\mu_2,(1-t)\nu_1+t\nu_2)\leq (1-t)W_1(\mu_1,\nu_1)+tW_1(\mu_2,\nu_2).
\end{equation}
\end{proposition}
\begin{proof}
Let $\omega_1\in\Pi(\mu_1,\nu_1), \omega_2\in\Pi(\mu_2,\nu_2)$ where $\Pi(\mu,\nu)\subseteq\mathcal{P}(\mathbb{P}\times\mathbb{P})$ denote the set of all couplings of $\mu,\nu\in\mathcal{P}^1(\mathbb{P})$. Then $(1-t)\omega_1+t\omega_2\in\Pi((1-t)\mu_1+t\mu_2,(1-t)\nu_1+t\nu_2)$ and we have
\begin{equation*}
\begin{split}
W_1&((1-t)\mu_1+t\mu_2,(1-t)\nu_1+t\nu_2)\\
&=\inf_{\gamma\in\Pi((1-t)\mu_1+t\mu_2,(1-t)\nu_1+t\nu_2)}\int_{\mathbb{P}\times\mathbb{P}}d_\infty(A,B)d\gamma(A,B)\\
&\leq\int_{\mathbb{P}\times\mathbb{P}}d_\infty(A,B)d((1-t)\omega_1+t\omega_2)(A,B)\\
&=(1-t)\int_{\mathbb{P}\times\mathbb{P}}d_\infty(A,B)d\omega_1(A,B)+t\int_{\mathbb{P}\times\mathbb{P}}d_\infty(A,B)d\omega_2(A,B),
\end{split}
\end{equation*}
thus by taking infima in $\omega_1\in\Pi(\mu_1,\nu_1), \omega_2\in\Pi(\mu_2,\nu_2)$ \eqref{eq:P:W_1convex} follows.
\end{proof}

The following two results are proved in \cite{limpalfia3}, we provide their proofs here for completeness. The first one is also obtained very recently in \cite{lawson2} through approximation by finitely supported probability measures.

\begin{theorem}[cf. \cite{limpalfia3}]\label{T:L1KarcherUniqueness}
Assume $\mu\in\mathcal{P}^1(\mathbb{P})$. Then the Karcher equation \eqref{eq:D:Karcher} has a unique solution in $\mathbb{P}$.
\end{theorem}
\begin{proof}
Let $X\in\mathbb{P}$ be a solution of \eqref{eq:D:Karcher}, i.e.
\begin{equation*}
\int_{\mathbb{P}}\log_{X}Ad\mu(A)=0.
\end{equation*}
Let $B(X,R):=\{Y\in\mathbb{P}:d_\infty(Y,X)<R\}$. Then since $\int_{\mathbb{P}}d_\infty(X,A)d\mu(A)<+\infty$ from Proposition 23 of Chapter 4 in \cite{royden} it follows that
\begin{equation}\label{eq0:T:L1KarcherUniqueness2}
\lim_{R\to\infty}\int_{\mathbb{P}\setminus B(X,R)}d_\infty(X,A)d\mu(A)=0.
\end{equation}
For $R\in[0,\infty)$, if $\mu(\mathbb{P}\setminus B(X,R))>0$ define
\begin{equation*}
E(R):=\frac{1}{\mu(\mathbb{P}\setminus B(X,R))}\int_{\mathbb{P}\setminus B(X,R)}\log(X^{-1/2}AX^{-1/2})d\mu(A)
\end{equation*}
and $E(R):=0$ otherwise. Also define $Z(R):=X^{1/2}\exp(E(R))X^{1/2}$ and $\mu_R\in\mathcal{P}^1(\mathbb{P})$ by
\begin{equation*}
\mu_R:=\mu|_{B(X,R)}+\mu(\mathbb{P}\setminus B(X,R))\delta_{Z(R)}
\end{equation*}
where $\mu|_{B(X,R)}$ is the restriction of $\mu$ to $B(X,R)$. Note that $\mu_R$ has bounded support for any $R\in(0,\infty)$.

Next, we claim that $\lim_{R\to\infty}W_1(\mu_R,\mu)=0$. If $W_1(\mu_{R_0},\mu)=0$ for some $R_0>0$ then $W_1(\mu_R,\mu)=0$ for all $R\geq R_0$ and we are done, so assume $W_1(\mu_{R},\mu)\neq 0$. We have
\begin{equation*}
\begin{split}
W_1(\mu_R,\mu)&=W_1\left(\mu|_{B(X,R)}+\mu(\mathbb{P}\setminus B(X,R))\delta_{Z(R)},\right.\\
&\left.\quad\quad\quad\mu|_{B(X,R)}+\mu(\mathbb{P}\setminus B(X,R))\frac{1}{\mu(\mathbb{P}\setminus B(X,R))}\mu|_{\mathbb{P}\setminus B(X,R)}\right)\\
&\leq\mu(B(X,R))W_1\left(\frac{1}{\mu(B(X,R))}\mu|_{B(X,R)},\frac{1}{\mu(B(X,R))}\mu|_{B(X,R)}\right)\\
&\quad+\mu(\mathbb{P}\setminus B(X,R))W_1\left(\delta_{Z(R)},\frac{1}{\mu(\mathbb{P}\setminus B(X,R))}\mu|_{\mathbb{P}\setminus B(X,R)}\right)\\
&=\int_{\mathbb{P}\setminus B(X,R)}d_\infty(Z(R),A)d\mu(A)\\
&\leq\int_{\mathbb{P}\setminus B(X,R)}d_\infty(Z(R),X)+d_\infty(X,A)d\mu(A)\\
&=\int_{\mathbb{P}\setminus B(X,R)}\|E(R)\|d\mu(A)+\int_{\mathbb{P}\setminus B(X,R)}d_\infty(X,A)d\mu(A)\\
&=\left\|\int_{\mathbb{P}\setminus B(X,R)}\log(X^{-1/2}AX^{-1/2})d\mu(A)\right\|+\int_{\mathbb{P}\setminus B(X,R)}d_\infty(X,A)d\mu(A)\\
&\leq\int_{\mathbb{P}\setminus B(X,R)}\left\|\log(X^{-1/2}AX^{-1/2})\right\|d\mu(A)+\int_{\mathbb{P}\setminus B(X,R)}d_\infty(X,A)d\mu(A)\\
&=2\int_{\mathbb{P}\setminus B(X,R)}d_\infty(X,A)d\mu(A)
\end{split}
\end{equation*}
where to obtain the first inequality we used \eqref{eq:P:W_1convex}. This proves our claim by \eqref{eq0:T:L1KarcherUniqueness2}.

On one hand, since $\mu_R$ has bounded support for all $R\in(0,\infty)$ by Theorem~\ref{T:KarcherExist2} it follows that the Karcher equation
\begin{equation}\label{eq1:T:L1KarcherUniqueness2}
\int_{\mathbb{P}}\log_YAd\mu_R(A)=0
\end{equation}
has a unique solution in $\mathbb{P}$ and that must be $\Lambda(\mu_R)$ by Theorem~\ref{T:LambdaExists}. On the other hand, we have that by definition $X$ is also a solution of \eqref{eq1:T:L1KarcherUniqueness2}, thus $\Lambda(\mu_R)=X$ for all $R\in(0,\infty)$. Now by Proposition~\ref{P:weakW1agree} we choose a sequence of finitely supported probability measures $\mu_n\in\mathcal{P}^1(\mathbb{P})$ that is $W_1$-converging to $\mu$, so by Theorem~\ref{T:LambdaExists} $\Lambda(\mu_n)\to\Lambda(\mu)$. Then, by the claim $W_1(\mu_R,\mu_n)\to 0$ as $R,n\to \infty$, thus by the contraction property \eqref{eq:T:LambdaExists} $d_\infty(\Lambda(\mu_R),\Lambda(\mu_n))\to 0$, that is $d_\infty(X,\Lambda(\mu_n))\to 0$ and also $\Lambda(\mu_n)\to\Lambda(\mu)$ proving that $X=\Lambda(\mu)$, thus the uniqueness of the solution of \eqref{eq:D:Karcher}.
\end{proof}

\begin{proposition}[cf. \cite{limpalfia3}]\label{P:StationaryFlow}
Let $\mu\in\mathcal{P}^1(\mathbb{P})$. Then the semigroup $S(t)\Lambda(\mu)$ generated in Theorem~\ref{T:ExponentialFormula} is stationary, that is $S(t)\Lambda(\mu)=\Lambda(\mu)$ for all $t>0$.
\end{proposition}
\begin{proof}
It is enough to show that $J^\mu_\lambda(\Lambda(\mu))=\Lambda(\mu)$ for any $\lambda>0$. Indeed by substitution $Z=\Lambda(\mu)$ is a solution of
\begin{equation*}
\frac{\lambda}{\lambda+1}\int_{\mathbb{P}}\log_{Z}Ad\mu(A)+\frac{1}{\lambda+1}\log_{Z}(\Lambda(\mu))=0
\end{equation*}
but this solution is unique by Theorem~\ref{T:L1KarcherUniqueness} and by definition \eqref{eq:D:resolvent} it is $J^\mu_\lambda(\Lambda(\mu))$.
\end{proof}

\begin{corollary}\label{C:ODECurvesConverge}
Let $\mu\in\mathcal{P}^1(\mathbb{P})$ and $X\in\mathbb{P}$. Then
\begin{equation}\label{eq:C:ODECurvesConverge}
d_\infty(S(t)X,\Lambda(\mu))\leq e^{-t}d_\infty(X,\Lambda(\mu))
\end{equation}
for all $t>0$.
\end{corollary}
\begin{proof}
By the previous Proposition~\label{P:StationaryFlow} we have that $S(t)\Lambda(\mu)=\Lambda(\mu)$ for all $t>0$. Thus, combined with \eqref{eq2:T:ExponentialFormula} we get
\begin{equation*}
d_\infty(S(t)X,\Lambda(\mu))=d_\infty(S(t)X,S(t)\Lambda(\mu))\leq e^{-t}d_\infty(X,\Lambda(\mu)).
\end{equation*}
\end{proof}

\section{Nodice theorem through resolvent iterations}

The first result of this section is about the convergence of resolvent iterations towards $\Lambda$. It can be viewed as a nonlinear proximal point algorithm.

\begin{proposition}\label{P:Proximal}
Let $\mu\in\mathcal{P}^1(\mathbb{P})$, $d\geq 0$ an integer and $X\in\mathbb{P}$. Let $X_0:=X$ and define $X_{k+1}:=J^\mu_{1/(k+d)}(X_k)$ for $k\in\mathbb{N}$. Then $d_\infty(X_k,\Lambda(\mu))\to 0$.
\end{proposition}
\begin{proof}
We use the notation of Theorem~\ref{T:KobayashiEst}. Let $t_k:=\sum_{i=d+1}^k\frac{1}{i}$ and $\hat{t}_j:=j\frac{t_n}{n}$ for a fixed $n\in\mathbb{N}$ such that $\lfloor n/t_n\rfloor\geq d$, so that $\tau_j=\frac{1}{j+d}$ and $\hat{\tau}=\hat{\tau}_j=\frac{t_n}{n}$ for all $1\leq j\leq n$. Let $S(t)$ denote the semigroup generated by $J^\mu$ in Theorem~\ref{T:ExponentialFormula}. Then
\begin{equation*}
\begin{split}
d_\infty(X_n,\Lambda(\mu))&\leq d_\infty(X_n,\left(J^\mu_{\hat{\tau}}\right)^n(X))+d_\infty(\left(J^\mu_{\hat{\tau}}\right)^n(X),S(t_n)X)\\
&\quad+d_\infty(S(t_n)X,\Lambda(\mu))\\
&\leq \left(1+\frac{t_n}{n}\right)^{-\lfloor n/t_n\rfloor+d}\prod_{j=\lfloor n/t_n\rfloor-d}^n(1+\tau_j)^{-1}\left[\left(t_n-n\frac{t_n}{n}\right)^2\right.\\
&\quad\left.+\sum_{j=1}^n\left(\frac{1}{(j+d)^2}+\frac{t_n^2}{n^2}\right)\right]^{1/2}C+t_n\left(\frac{1}{n}\right)^{1/2}C+e^{-t_n}d_\infty(X,\Lambda(\mu))\\
&\leq \left(1+\frac{t_n}{n}\right)^{-\lfloor n/t_n\rfloor+d}\frac{\lfloor n/t_n\rfloor-d}{n+1}\sqrt{\frac{\pi^2}{6}+\frac{t_n^2}{n}}C+\frac{t_n}{n^{1/2}}C\\
&\quad+e^{-t_n}d_\infty(X,\Lambda(\mu))\\
&=O(1/\log(n))
\end{split}
\end{equation*}
where we used \eqref{eq:T:KobayashiEst}, \eqref{eq4:T:ExponentialFormula}, \eqref{eq2:T:ExponentialFormula} and \eqref{eq:C:ODECurvesConverge} to obtain the second inequality, then we used the formula $\sum_{j=1}^\infty\frac{1}{j^2}=\frac{\pi^2}{6}$ and the fact that $t_n=O(\log(n))$ to obtain the last two inequalities. The above bound proves the assertion.
\end{proof}

We need an elementary lemma from \cite[Lemma~3.4]{Bert1} for later use.

\begin{lemma}\label{L:2}
Let $a_k,b_k,c_k\geq 0$ be sequences such that $a_{k+1}\leq a_k-b_k+c_k$ for any $k \ge 1$,
and assume $\sum_{k=1}^{\infty}c_k<\infty$.
Then the sequence $a_k$ converges and also $\sum_{k=1}^{\infty}b_k<\infty$.
\end{lemma}

Let us quote a lemma from \cite{Ne1}:
\begin{lemma}\label{L:convrate}
Let $a_k\geq 0$ be a sequence such that
\[ a_{k+1}\leq \left(1-\frac{\alpha}{k+1}\right)a_k+\frac{\beta}{(k+1)^2}, \]
where $\alpha,\beta>0$. Then
\[ a_{k}\leq
\begin{cases}
\frac{1}{(k+2)^{\alpha}}\left(a_0+\frac{2^\alpha\beta(2-\alpha)}{1-\alpha}\right) & \mbox{if }0<\alpha<1;\\
\frac{\beta(1+\log(k+1))}{k+1} & \mbox{if }\alpha=1;\\
\frac{1}{(\alpha-1)(k+2)}\left(\beta+\frac{(\alpha-1)a_0-\beta}{(k+2)^{\alpha-1}}\right) & \mbox{if }\alpha>1.
\end{cases} \]
\end{lemma}

The next two lemmas gives us technical tools to control error estimates occurring in resolvent iterations. The first one is a Lipschitz estimate for the relative operator entropy.

\begin{lemma}\label{L:LipschitzEntropy}
Let $A,X,Y\in\mathbb{P}$. Then
\begin{equation}\label{eq:L:LipschitzEntropy}
\|\log_XA-\log_YA\|\leq O(d_\infty(X,Y))\left(e^{d_\infty(I,Y)+d_\infty(I,A)}+d_\infty(X,A)e^{2d_\infty(I,X)}\right)
\end{equation}
\end{lemma}
\begin{proof}
By the continuous functional calculus we have that $$\log_XA=X^{1/2}\log(X^{-1/2}AX^{-1/2})X^{1/2}=X\log(X^{-1}A),$$ thus we have
\begin{equation*}
\begin{split}
\|\log_XA-&\log_YA\|=\|X\log(X^{-1}A)-Y\log(Y^{-1}A)\|\\
\leq &\|(X-Y)\log(X^{-1}A)\|+\|Y\left[\log(X^{-1}A)-\log(Y^{-1}A)\right]\|\\
\leq &\|(X-Y)\|\|X^{-1/2}\|\|\log(X^{-1/2}AX^{-1/2})\|\|X^{1/2}\|\\
&+\|Y\|\|A^{-1/2}\|\|\log(A^{1/2}X^{-1}A^{1/2})-\log(A^{1/2}Y^{-1}A^{1/2})\|\|A^{1/2}\|.
\end{split}
\end{equation*}
Then using \eqref{eq:order}, \eqref{eq:invariance}, \eqref{eq:EMI}, that $\|X^{\pm 1}\|\leq e^{d_\infty(I,X)}$ and $\|X^{1/2}\|=\|X\|^{1/2}$, we estimate the above further as
\begin{equation*}
\begin{split}
\|\log_XA-&\log_YA\|\\
\leq &\|X\|\|X^{-1/2}YX^{-1/2}-I\|\|X^{-1/2}\|\|X^{1/2}\|d_\infty(X,A)\\
&+\|Y\|\|A^{-1/2}\|d_\infty(A^{1/2}X^{-1}A^{1/2},A^{1/2}Y^{-1}A^{1/2})\|A^{1/2}\|\\
\leq &d_\infty(X,A)e^{2d_\infty(I,X)}(e^{d_\infty(X,Y)}-1)+e^{d_\infty(I,Y)+d_\infty(I,A)}d_\infty(X,Y).
\end{split}
\end{equation*}
Here we also used that $\|X^{-1/2}YX^{-1/2}-I\|\leq \max\{e^{d_\infty(X,Y)}-1,1-e^{-d_\infty(X,Y)}\}$ by \eqref{eq:order} and then that $e^{d_\infty(X,Y)}-1\geq 1-e^{-d_\infty(X,Y)}$, since $(e^{d_\infty(X,Y)}-1)^2\geq 0$. Now from the above, the assertion follows from $e^{d_\infty(X,Y)}-1\leq O(d_\infty(X,Y))$.
\end{proof}

\begin{lemma}\label{L:resolventAppr}
Let $X,Y\in\mathbb{P}$ and $E=E^*\in\mathcal{B}(\mathcal{H})$. Then there exists $\hat{X}\in\mathbb{P}$ such that
\begin{equation*}
X-Y+E=\log_Y\hat{X},
\end{equation*}
and we also have that
\begin{equation}\label{eq:L:resolventAppr}
\begin{split}
d_\infty(\hat{X},X)&\leq O\left(\left\|X^{-1/2}EX^{-1/2}\right\|\right)+O\left(e^{2d_\infty(X,Y)}\right)\times\\
&\times \left[O(d_\infty^2(X,Y))+O\left(\|Y^{-1/2}EY^{-1/2}\|^2\right)\right.\\
&\left.\quad+O\left(d_\infty(X,Y)\right)O\left(\|Y^{-1/2}EY^{-1/2}\|\right)\right]
\end{split}
\end{equation}
for small enough $d_\infty(X,Y),\left\|X^{-1/2}EX^{-1/2}\right\|,\|Y^{-1/2}EY^{-1/2}\|\geq 0$.
\end{lemma}
\begin{proof}
First of all, notice that
\begin{equation}\label{eq1:L:resolventAppr}
\|X^{-1/2}Y^{1/2}\|^2=\sup_{v\in \mathcal{H},|v|=1}v^*Y^{1/2}X^{-1}Y^{1/2}v\leq e^{d_\infty(X,Y)}
\end{equation}
and
\begin{equation*}
e^{-d_\infty(X,Y)}-I\leq Y^{-1/2}XY^{-1/2}-I\leq e^{d_\infty(X,Y)}-I,
\end{equation*}
thus
\begin{equation}\label{eq2:L:resolventAppr}
\|Y^{-1/2}XY^{-1/2}-I\|\leq e^{d_\infty(X,Y)}-1\leq O(d_\infty(X,Y)),
\end{equation}
since $e^{d_\infty(X,Y)}-1\geq 1-e^{-d_\infty(X,Y)}$.
Now, from the assumption we get that
\begin{equation*}
\begin{split}
Y^{-1/2}XY^{-1/2}-I+Y^{-1/2}EY^{-1/2}&=\log(Y^{-1/2}\hat{X}Y^{-1/2})\\
\exp\left(Y^{-1/2}XY^{-1/2}-I+Y^{-1/2}EY^{-1/2}\right)&=Y^{-1/2}\hat{X}Y^{-1/2}\\
X^{-1/2}Y^{1/2}\exp\left(Y^{-1/2}XY^{-1/2}-I+Y^{-1/2}EY^{-1/2}\right)\times&\\
\times Y^{1/2}X^{-1/2}&=X^{-1/2}\hat{X}X^{-1/2}
\end{split}
\end{equation*}
establishing the existence of $\hat{X}\in\mathbb{P}$. By considering the Taylor expansion \eqref{eq:L:SeriesExpansions1}, we continue as
\begin{equation*}
\begin{split}
X^{-1/2}&\hat{X}X^{-1/2}=I+X^{-1/2}EX^{-1/2}\\
&+X^{-1/2}Y^{1/2}O\left(\|Y^{-1/2}XY^{-1/2}-I+Y^{-1/2}EY^{-1/2}\|^2\right)Y^{1/2}X^{-1/2},
\end{split}
\end{equation*}
so it follows, using the power series expansion \eqref{eq:L:SeriesExpansions2}, that
\begin{equation*}
\begin{split}
d_\infty(\hat{X},X)&=\|\log(X^{-1/2}\hat{X}X^{-1/2})\|\\
&\leq O\left[\left\|X^{-1/2}EX^{-1/2}+X^{-1/2}Y^{1/2}\times\right.\right.\\
&\left.\left.\quad\times O\left(\|Y^{-1/2}XY^{-1/2}-I+Y^{-1/2}EY^{-1/2}\|^2\right)Y^{1/2}X^{-1/2}\right\|\right]\\
&\leq O\left[\left\|X^{-1/2}EX^{-1/2}\right\|+\left\|X^{-1/2}Y^{1/2}\right\|^2\times\right.\\
&\left.\quad\times O\left(\|Y^{-1/2}XY^{-1/2}-I+Y^{-1/2}EY^{-1/2}\|^2\right)\right]\\
&\leq O\left(\left\|X^{-1/2}EX^{-1/2}\right\|\right)+O\left(\left\|X^{-1/2}Y^{1/2}\right\|^2\right)\times\\
&\quad\times O\left(\|Y^{-1/2}XY^{-1/2}-I+Y^{-1/2}EY^{-1/2}\|^2\right)\\
&\leq O\left(\left\|X^{-1/2}EX^{-1/2}\right\|\right)+O\left(\left\|X^{-1/2}Y^{1/2}\right\|^2\right)\times\\
&\quad\times O\left(\|Y^{-1/2}XY^{-1/2}-I\|^2+\|Y^{-1/2}EY^{-1/2}\|^2\right.\\
&\quad\left.+2\|Y^{-1/2}XY^{-1/2}-I\|\|Y^{-1/2}EY^{-1/2}\|\right)
\end{split}
\end{equation*}
Now using \eqref{eq1:L:resolventAppr} and \eqref{eq2:L:resolventAppr}, we get from the above that
\begin{equation*}
\begin{split}
&d_\infty(\hat{X},X)\leq O\left(\left\|X^{-1/2}EX^{-1/2}\right\|\right)+O\left(e^{2d_\infty(X,Y)}\right)\times\\
&\times \left[O(d_\infty^2(X,Y))+O\left(\|Y^{-1/2}EY^{-1/2}\|^2\right)+O\left(d_\infty(X,Y)\|Y^{-1/2}EY^{-1/2}\|\right)\right].
\end{split}
\end{equation*}
From here the assertion follows.
\end{proof}

Now we are in position to prove Holbrook's nodice theorem for the Karcher mean. Its proof gives a new proof of the result in the matrix case given by Holbrook \cite{Hol}.

\begin{theorem}[Nodice]\label{T:Nodice}
Let $\frac{1}{k}\sum_{i=1}^{k}\delta_{A_i}=:\mu\in\mathcal{P}^1(\mathbb{P})$ for a fixed integer $k$. Let $S_1:=A_1$ and $S_{n+1}:=S_n\#_{\frac{1}{n+1}}A_{\overline{n+1}}$, where $\overline{n}$ is defined to equal the residual of $n \mod k$ and the $0$ residual is identified with $k$. Then $S_k\to \Lambda(\mu)$ in $d_\infty$.
\end{theorem}
\begin{proof}
The idea of the proof is to compare the sequence $\left\{S_{nk}\right\}_{n\in\mathbb{N}}$ with the sequence produced by the resolvent iteration $\hat{S}_n$ converging to $\Lambda(\mu)$ in Proposition~\ref{P:Proximal} with some well chosen starting point $X_0\in\mathbb{P}$.

First of all, notice that the sequence $S_n$ is bounded. Indeed, by
\begin{equation*}
\mathrm{diam}(\supp(\mu))=\max_{1\leq i,j\leq k}d_\infty(A_i,A_j),
\end{equation*}
we have that $A_i\in B(A_1,\mathrm{diam}(\supp(\mu)))$ for all $1\leq i\leq k$ and by induction we obtain that $S_n\in B(A_1,\mathrm{diam}(\supp(\mu)))$ as well, since $d_\infty(X,X\#_tY)=(1-t)d_\infty(X,Y)$ for $t\in[0,1]$. In other words for all $n\geq 1$
\begin{equation}\label{eq2:T:Nodice}
\left\|S_{n}^{\pm 1}\right\|=e^{d_\infty(S_{n},I)}\leq e^{d_\infty(A_1,I)+2\mathrm{diam}(\supp(\mu))}.
\end{equation}
and thus
\begin{equation}\label{eq3:T:Nodice}
d_\infty(S_{n},S_{n+1})\leq \frac{2}{n+1}\mathrm{diam}(\supp(\mu)).
\end{equation}

Let $1\leq i\leq k$ be arbitrary and $N\in\mathbb{N}$ such that $\frac{2}{Nk}\mathrm{diam}(\supp(\mu))<1$. Let $n\geq N$ be an integer. By definition $S_{nk+i+1}$ satisfies the Karcher equation, that is
\begin{equation*}
\frac{1}{nk+i+1}\log_{S_{nk+i+1}}A_{\overline{nk+i+1}}+\log_{S_{nk+i+1}}(S_{nk+i})=0.
\end{equation*}
From the above by expanding $\log$ into a Taylor series according to Lemma~\ref{L:ResolventAsymptotics}, we get
\begin{equation*}
\begin{split}
\frac{1}{nk+i+1}\log_{S_{nk+i+1}}&A_{\overline{nk+i+1}}+S_{nk+i}-S_{nk+i+1}\\
&+S_{nk+i+1}^{1/2}O(d_\infty(S_{nk+i},S_{nk+i+1})^2)S_{nk+i+1}^{1/2}=0,
\end{split}
\end{equation*}
and using \eqref{eq2:T:Nodice} and \eqref{eq3:T:Nodice} we get
\begin{equation*}
\begin{split}
\frac{1}{nk+i+1}\log_{S_{nk+i+1}}&A_{\overline{nk+i+1}}+S_{nk+i}-S_{nk+i+1}\\
&+O\left(4\frac{\mathrm{diam}(\supp(\mu))^2}{(nk+i+1)^2}\right)e^{d_\infty(A_1,I)+2\mathrm{diam}(\supp(\mu))}=0.
\end{split}
\end{equation*}
Summing the above identity in $0\leq i\leq k-1$, we get
\begin{equation*}
\begin{split}
S_{nk}-S_{k(n+1)}&+\sum_{i=0}^{k-1}\frac{1}{nk+i+1}\log_{S_{nk+i+1}}A_{\overline{nk+i+1}}\\
&+O\left(4\frac{\mathrm{diam}(\supp(\mu))^2}{(nk+i+1)^2}\right)e^{d_\infty(A_1,I)+2\mathrm{diam}(\supp(\mu))}=0.
\end{split}
\end{equation*}
From the above it follows that
\begin{equation}\label{eq4:T:Nodice}
\begin{split}
S_{nk}-&S_{k(n+1)}+O\left(4\frac{\mathrm{diam}(\supp(\mu))^2}{(nk)^2}\right)ke^{d_\infty(A_1,I)+2\mathrm{diam}(\supp(\mu))}\\
&+\sum_{i=1}^{k}\frac{1}{nk+i}\log_{S_{nk+i}}A_{\overline{nk+i}}=0,\\
S_{nk}-&S_{k(n+1)}+O\left(4\frac{\mathrm{diam}(\supp(\mu))^2}{(nk)^2}\right)ke^{d_\infty(A_1,I)+2\mathrm{diam}(\supp(\mu))}\\
&+\sum_{i=1}^{k}\frac{1}{k(n+1)}\log_{S_{nk+i}}A_{\overline{nk+i}}+\left(\frac{1}{nk+i}-\frac{1}{k(n+1)}\right)\log_{S_{nk+i}}A_{\overline{nk+i}}=0,\\
S_{nk}-&S_{k(n+1)}+O\left(4\frac{\mathrm{diam}(\supp(\mu))^2}{(nk)^2}\right)ke^{d_\infty(A_1,I)+2\mathrm{diam}(\supp(\mu))}\\
&+\sum_{i=1}^{k}\frac{1}{k(n+1)}\log_{S_{nk+i}}A_{\overline{nk+i}}+\frac{k-i}{k(n+1)(nk+i)}\log_{S_{nk+i}}A_{\overline{nk+i}}=0,\\
S_{nk}-&S_{k(n+1)}+O\left(4\frac{\mathrm{diam}(\supp(\mu))^2}{(nk)^2}\right)ke^{d_\infty(A_1,I)+2\mathrm{diam}(\supp(\mu))}\\
&+\sum_{i=1}^{k}\frac{1}{k(n+1)}\log_{S_{nk+i}}A_{\overline{nk+i}}+\frac{k-i}{k(n+1)(nk+i)}\log_{S_{nk+i}}A_{\overline{nk+i}}=0.
\end{split}
\end{equation}
Now we estimate the norm of the last term above using \eqref{eq2:T:Nodice} as
\begin{equation*}
\begin{split}
\frac{k-i}{k(n+1)(nk+i)}&\|\log_{S_{nk+i}}A_{\overline{nk+i}}\|\leq \frac{k-i}{k(n+1)(nk+i)}\|S_{nk+i}^{\pm 1}\|d_\infty(S_{nk+i},A_{\overline{nk+i}})\\
&\leq \frac{d_\infty(S_{nk+i},A_{\overline{nk+i}})(k-i)}{k(n+1)(nk+i)}e^{d_\infty(A_1,I)+2\mathrm{diam}(\supp(\mu))}\\
&\leq \frac{2\mathrm{diam}(\supp(\mu))k}{k(n+1)(nk+i)}e^{d_\infty(A_1,I)+2\mathrm{diam}(\supp(\mu))}\\
&\leq \frac{2\mathrm{diam}(\supp(\mu))}{n^2k}e^{d_\infty(A_1,I)+2\mathrm{diam}(\supp(\mu))},
\end{split}
\end{equation*}
so combined with the last equation in \eqref{eq4:T:Nodice}, we get
\begin{equation}\label{eq5:T:Nodice}
\begin{split}
S_{nk}-&S_{k(n+1)}+O\left(6\frac{\mathrm{diam}(\supp(\mu))^2}{n^2k}\right)ke^{d_\infty(A_1,I)+2\mathrm{diam}(\supp(\mu))}\\
&+\sum_{i=1}^{k}\frac{1}{k(n+1)}\log_{S_{nk+i}}A_{\overline{nk+i}}=0,\\
S_{nk}-&S_{k(n+1)}+O\left(6\frac{\mathrm{diam}(\supp(\mu))^2}{n^2k}\right)ke^{d_\infty(A_1,I)+2\mathrm{diam}(\supp(\mu))}\\
&+\sum_{i=1}^{k}\frac{1}{k(n+1)}\log_{S_{k(n+1)}}A_{\overline{nk+i}}\\
&+\frac{1}{k(n+1)}\left(\log_{S_{nk+i}}A_{\overline{nk+i}}-\log_{S_{k(n+1)}}A_{\overline{nk+i}}\right)=0.
\end{split}
\end{equation}
Here again, we estimate the norm of the last term using the Lipschitz estimate of Lemma~\ref{L:LipschitzEntropy}, the triangle inequality for $d_\infty$, \eqref{eq2:T:Nodice} and \eqref{eq3:T:Nodice} as
\begin{equation*}
\begin{split}
&\left\|\log_{S_{nk+i}}A_{\overline{nk+i}}-\log_{S_{k(n+1)}}A_{\overline{nk+i}}\right\|\leq  O(d_\infty(S_{k(n+1)},S_{nk+i}))\times \\
&\quad\quad\times\left(e^{d_\infty(I,S_{nk+i})+d_\infty(I,A_{\overline{nk+i}})}+d_\infty(S_{n(k+1)},A_{\overline{nk+i}})e^{2d_\infty(I,S_{n(k+1)})}\right)\\
&\quad\leq O\left(\sum_{j=i+1}^k\frac{2\mathrm{diam}(\supp(\mu))}{nk+j}\right)\left(e^{2d_\infty(A_1,I)+4\mathrm{diam}(\supp(\mu))}\right.\\
&\quad\quad\left.+2\mathrm{diam}(\supp(\mu))e^{2d_\infty(A_1,I)+4\mathrm{diam}(\supp(\mu))}\right)\\
&\quad\leq O\left(\frac{2\mathrm{diam}(\supp(\mu))k}{nk}\right)\left(e^{2d_\infty(A_1,I)+4\mathrm{diam}(\supp(\mu))}\right.\\
&\quad\quad\left.+2\mathrm{diam}(\supp(\mu))e^{2d_\infty(A_1,I)+4\mathrm{diam}(\supp(\mu))}\right)\\
&\quad\leq O\left(\frac{2\mathrm{diam}(\supp(\mu))}{n}\right)\left(e^{2d_\infty(A_1,I)+4\mathrm{diam}(\supp(\mu))}\right.\\
&\quad\quad\left.+2\mathrm{diam}(\supp(\mu))e^{2d_\infty(A_1,I)+4\mathrm{diam}(\supp(\mu))}\right).
\end{split}
\end{equation*}
Thus using the above estimate, the last equation of \eqref{eq5:T:Nodice} implies
\begin{equation*}
\begin{split}
S_{nk}-S_{k(n+1)}&+O\left(6\frac{\mathrm{diam}(\supp(\mu))^2}{n^2k}\right)ke^{d_\infty(A_1,I)+2\mathrm{diam}(\supp(\mu))}\\
&+\frac{k}{k(n+1)}O\left(\frac{2\mathrm{diam}(\supp(\mu))}{n}\right)\left(e^{2d_\infty(A_1,I)+4\mathrm{diam}(\supp(\mu))}\right.\\
&\left.+2\mathrm{diam}(\supp(\mu))e^{2d_\infty(A_1,I)+4\mathrm{diam}(\supp(\mu))}\right)\\
&+\sum_{i=1}^{k}\frac{1}{k(n+1)}\log_{S_{k(n+1)}}A_{\overline{nk+i}}=0.
\end{split}
\end{equation*}
Using the properties of $O(\cdot)$ and assuming $\mathrm{diam}(\supp(\mu))>1$ without loss of generality, the last identity implies
\begin{equation}\label{eq6:T:Nodice}
\begin{split}
S_{nk}-S_{k(n+1)}&+O\left(8\frac{\mathrm{diam}(\supp(\mu))^2}{n^2}\right)\left(e^{2d_\infty(A_1,I)+4\mathrm{diam}(\supp(\mu))}\right.\\
&\left.+2\mathrm{diam}(\supp(\mu))e^{2d_\infty(A_1,I)+4\mathrm{diam}(\supp(\mu))}\right)\\
&+\sum_{i=1}^{k}\frac{1}{k(n+1)}\log_{S_{k(n+1)}}A_{\overline{nk+i}}=0,
\end{split}
\end{equation}
for large enough $n>0$ such that $8\frac{\mathrm{diam}(\supp(\mu))^2}{n^2}<1$. Notice that \eqref{eq3:T:Nodice} ensures that
\begin{equation}\label{eq7:T:Nodice}
d_\infty(S_{nk},S_{k(n+1)})\leq \frac{2\mathrm{diam}(\supp(\mu))}{n},
\end{equation}
in particular for large $n>0$ we also have that the norm of the quantity
\begin{equation*}
\begin{split}
E:=&O\left(8\frac{\mathrm{diam}(\supp(\mu))^2}{n^2}\right)\left(e^{2d_\infty(A_1,I)+4\mathrm{diam}(\supp(\mu))}\right.\\
&\left.+2\mathrm{diam}(\supp(\mu))e^{2d_\infty(A_1,I)+4\mathrm{diam}(\supp(\mu))}\right)
\end{split}
\end{equation*}
is arbitrarily small, since the asymptotic error constants in $O(\cdot)$ only depend on the $A_i$, not on $n$ larger then the previously specified magnitude. Thus there exists an $N\in\mathbb{N}$ satisfying also $8\frac{\mathrm{diam}(\supp(\mu))^2}{N^2}<1$, such that by Lemma~\ref{L:resolventAppr} there exists an $\overline{S}_{n}\in\mathbb{P}$ such that
\begin{equation}\label{eq8:T:Nodice}
\frac{1}{n+1}\sum_{i=1}^{k}\frac{1}{k}\log_{S_{k(n+1)}}A_{\overline{nk+i}}+\log_{S_{k(n+1)}}\overline{S}_{n}=0
\end{equation}
and
\begin{equation}\label{eq9:T:Nodice}
d_\infty(S_{kn},\overline{S}_{n})\leq O\left(\frac{1}{n^2}\right)
\end{equation}
for all $n\geq N$, where we also used \eqref{eq2:T:Nodice}. In other words $S_{k(n+1)}=J_{\frac{1}{n+1}}^{\mu}\overline{S}_{n}$. Now let $\hat{S}_n:=\overline{S}_N$ for $1\leq n\leq N$, and $\hat{S}_{n+1}:=J_{\frac{1}{n+1}}^{\mu}\hat{S}_{n}$ for $n\geq N$ recursively. Then using the contraction property \eqref{eq:P:ResolventContraction} and \eqref{eq9:T:Nodice} we get
\begin{equation*}
\begin{split}
d_\infty(S_{k(n+1)},\hat{S}_{n+1})&\leq \frac{1}{1+\frac{1}{n+1}}d_\infty(\overline{S}_{n},\hat{S}_{n})\\
&\leq \frac{1}{1+\frac{1}{n+1}}\left(d_\infty(S_{kn},\hat{S}_{n})+d_\infty(S_{kn},\overline{S}_{n})\right)\\
&=\left(1-\frac{1}{n+2}\right)\left[d_\infty(S_{kn},\hat{S}_{n})+O\left(\frac{1}{n^2}\right)\right]\\
&\leq\left(1-\frac{1}{n+2}\right)d_\infty(S_{kn},\hat{S}_{n})+O\left(\frac{1}{(n+2)^2}\right).
\end{split}
\end{equation*}
Denoting $a_{n+2}:=d_\infty(S_{k(n+1)},\hat{S}_{n+1})$ by Lemma~\ref{L:convrate} it follows that $a_n\to 0$. In particular, since $\hat{S}_n\to\Lambda(\mu)$ in $d_\infty$ by Proposition~\ref{P:Proximal}, the assertion is proved for the subsequence $\{S_{kn}\}_{n\in\mathbb{N}}$. The convergence of the rest of the sequence $S_n$ follows from the estimates
\begin{equation*}
\begin{split}
d_\infty(S_{nk},S_{kn+i})&\leq \sum_{j=0}^i\frac{2\mathrm{diam}(\supp(\mu))}{nk+j}\\
&\leq k\frac{2\mathrm{diam}(\supp(\mu))}{nk}\\
&=\frac{2\mathrm{diam}(\supp(\mu))}{n}
\end{split}
\end{equation*}
valid for any $1\leq i\leq k$.
\end{proof}

\section{Convergence of stochastic discrete-time evolution systems}
In this section we prove the following nonlinear (Sturm's) $L^1$-strong law of large numbers for $\Lambda$.

\begin{theorem}\label{T:Sturmslln}
Let $\mu\in\mathcal{P}^1(\mathbb{P})$ and let $Y_i$ be an i.i.d. sequence of random variables with law $\mu$. Let $S_1:=Y_1$ and $S_{n+1}:=S_n\#_{\frac{1}{n+1}}Y_{n+1}$. Then $S_n\to\Lambda(\mu)$ a.s. in $d_\infty$.
\end{theorem}

The proof will make use of some auxiliary results as follows.

\begin{lemma}\label{L:Truncate}
Let $\epsilon>0$ and $\mu\in\mathcal{P}^1(\mathbb{P})$. Then there exists an $R>0$ such that
\begin{equation*}
\limsup_{n\to\infty}d_\infty(X_n,X^R_n)<\epsilon
\end{equation*}
almost surely, where $X_1:=Y_1$, $X^R_1:=Y^R_1$ and recursively $X_{n+1}:=X_n\#_{\frac{1}{n+1}}Y_{n+1}$, $X^R_{n+1}:=X^R_n\#_{\frac{1}{n+1}}Y^R_{n+1}$, where $Y_n$ is an i.i.d. sequence of $\mathbb{P}$-valued random variables with law $\mu$ and
\begin{equation}\label{eq:L:Truncate}
Y^R_n(\omega):=\left\{
    \begin{array}{lr}
      Y_n(\omega) & \mathrm{if}\quad d_\infty(Y_n(\omega),\Lambda(\mu))<R, \\
      \Lambda(\mu) & \mathrm{if}\quad d_\infty(Y_n(\omega),\Lambda(\mu))\geq R.
    \end{array}
    \right.
\end{equation}
\end{lemma}
\begin{proof}
Consider the non-negative real valued random variable $$d(\omega):=d_\infty(Y_1(\omega),Y^R_1(\omega)).$$ It is clearly in $L^1(\mu)$, since by the triangle inequality 
$$d(\omega)\leq d_\infty(Y_1(\omega),\Lambda(\mu)))+d_\infty(Y^R_1(\omega),\Lambda(\mu))$$
 and $\int_{\mathbb{P}}d_\infty(\Lambda(\mu),A)d\mu(A)<\infty$, so $\mathbb{E}d(\omega)<\infty$. From the $L^1$-integrability of $d_\infty$ with respect to $\mu$ it follows that
$$\lim_{R\to\infty}\int_{d_\infty(\Lambda(\mu),A)\geq R}d_\infty(\Lambda(\mu),A)d\mu(A)=0,$$
thus $\mathbb{E}d(\omega):=\int_{\mathbb{P}}d(\omega)d\mu(\omega)\to 0$ as $R\to 0$. Thus we choose an $R>0$ such that $\mathbb{E}d(\omega)<\epsilon$. Now use the convexity in Proposition~\ref{P:KarcherW1contracts} recursively to obtain
\begin{equation*}
d_\infty(X_n(\omega),X^R_n(\omega))\leq\frac{1}{n}\sum_{i=1}^n d_\infty(Y_i(\omega),Y^R_i(\omega)).
\end{equation*}
From this inequality and the fact that $\mathbb{E}d(\omega)<\epsilon$, it follows from the $L^1$-strong law of large numbers for a non-negative real valued random variable with finite expectation \cite{dudley} applied to $d(\omega)$ that 
$$\limsup_{n\to\infty}d_\infty(X_n(\omega),X^R_n(\omega))<\epsilon$$ 
almost surely, proving the assertion.
\end{proof}

\begin{lemma}\label{L:empiricalIteration}
Let $\mu,\nu_i\in\mathcal{P}^1(\mathbb{P})$ for $i\in\mathbb{N}$, $l\geq 0$ an integer and $X_0,Y_0\in\mathbb{P}$. Let
\begin{equation*}
X_{k+1}:=J^\mu_{1/(l+k+1)}(X_k)\quad\text{and}\quad Y_{k+1}:=J^{\nu_{k+1}}_{1/(l+k+1)}(Y_k).
\end{equation*}
Then
\begin{equation*}
d_\infty(X_{k+1},Y_{k+1})\leq \frac{l+1}{k+l+1}d_\infty(X_0,Y_0)+\frac{k-l}{k+l+1}\sum_{i=l+1}^{k+1}\frac{W_1(\mu,\nu_i)}{k-l}.
\end{equation*}
\end{lemma}
\begin{proof}
By \eqref{eq:T:LambdaExists} and Proposition~\ref{P:W_1convex} for any $\lambda>0$, $x,y\in\mathbb{P}$ and $\mu,\eta\in\mathcal{P}^1(\mathbb{P})$ we have that
\begin{equation*}
\begin{split}
d_\infty(J^\mu_\lambda(x),J^\eta_\lambda(y))&=d_\infty\left(\Lambda\left(\frac{1}{1+\lambda}\delta_x+\frac{\lambda}{1+\lambda}\mu\right),\Lambda\left(\frac{1}{1+\lambda}\delta_y+\frac{\lambda}{1+\lambda}\eta\right)\right)\\
&\leq W_1\left(\frac{1}{1+\lambda}\delta_x+\frac{\lambda}{1+\lambda}\mu,\frac{1}{1+\lambda}\delta_y+\frac{\lambda}{1+\lambda}\eta\right)\\
&\leq \frac{1}{1+\lambda}W_1(\delta_x,\delta_y)+\frac{\lambda}{1+\lambda}W_1(\mu,\eta)\\
&=\frac{1}{1+\lambda}d_\infty(x,y)+\frac{\lambda}{1+\lambda}W_1(\mu,\eta).
\end{split}
\end{equation*}
Thus when $\lambda=1/k$ we get that $d_\infty(J^\mu_{1/k}(x),J^\eta_{1/k}(y))\leq \frac{k}{k+1}d_\infty(x,y)+\frac{1}{k+1}W_1(\mu,\eta)$, so applying this inequality to $d_\infty(X_{k+1},Y_{k+1})$ and induction we get
\begin{equation*}
d_\infty(X_{k+1},Y_{k+1})\leq \frac{l+1}{k+l+1}d_\infty(X_0,Y_0)+\frac{1}{k+l+1}\sum_{i=l+1}^{k+1}W_1(\mu,\nu_i)
\end{equation*}
proving the assertion.
\end{proof}

\begin{lemma}\label{L:empiricalMeasureW1}
Let $\mu\in\mathcal{P}^1(\mathbb{P})$ and let $Y_n$ be a sequence of i.i.d. $\mathbb{P}$-valued random variables with law $\mu$. Then for any $\epsilon>0$ there exists an $N\in\mathbb{N}$ such that for any $n\geq N$
\begin{equation}\label{eq:L:empiricalMeasureW1}
\mathbb{E}W_1(\mu,\mu_n)<\epsilon,
\end{equation}
where $\mu_n\in\mathcal{P}^1(\mathbb{P})$ is a random measure defined as $\mu_n:=\sum_{i=1}^n\frac{1}{n}\delta_{Y_i}$.
\end{lemma}
\begin{proof}
The support $\supp(\mu)\subseteq\mathbb{P}$ is separable and $Y_n\in\supp(\mu)$, thus we can consider the separable metric space $(\supp(\mu),d_\infty)$, so that $\mu_n$ is the empirical measure of $\mu$ and thus $\mu_n$ is uniformly integrable almost surely. Thus we may apply Varadarajan's Theorem 11.4.1. in \cite{dudley} and conclude that $\mu_n\to\mu$ almost surely in the weak-$*$ topology, then appeal to Proposition~\ref{P:weakW1agree} that actually $\mu_n\to\mu$ in $W_1$ almost surely as well. Then for $\epsilon>0$ there exists $N\in\mathbb{N}$ such that for all $n\geq N$ we have $W_1(\mu,\mu_n)<\epsilon$ almost surely, so in particular \eqref{eq:L:empiricalMeasureW1} follows by Lebesgue's dominated convergence theorem and the measurability of $\omega\mapsto W_1(\mu,\mu_n(\omega))$ which holds by semi-continuity of $W_1$, see Remark 6.12. \cite{villani}.
\end{proof}

\begin{proof}[Proof of Theorem~\ref{T:Sturmslln}]
Let $\epsilon>0$. Then our goal is to find an $N\in\mathbb{N}$ such that for all $n\geq N$ we have that $d_\infty(S_n,\Lambda(\mu))<\epsilon$ almost surely.

First we pick an $R>0$ such that for $\mu^R:=\mu(B(\Lambda(\mu),R))\mu|_{B(\Lambda(\mu),R)}+(1-\mu(B(\Lambda(\mu),R)))\delta_{\Lambda(\mu)}$ we have that
\begin{equation}\label{eq1:T:Sturmslln}
d_\infty(\Lambda(\mu^r),\Lambda(\mu))\leq W_1(\mu,\mu^r)<\frac{\epsilon}{20},
\end{equation}
and we also assume that this $R$ is large enough so that by Lemma~\ref{L:Truncate}
\begin{equation*}
\limsup_{n\to\infty}d_\infty(S_n,S^{R}_n)<\frac{\epsilon}{20}
\end{equation*}
almost surely, where $S^R_1:=Y^R_1$ and $S^R_{n+1}:=S^R_n\#_{\frac{1}{n+1}}Y^R_{n+1}$ with
\begin{equation*}
Y^R_i(\omega):=\left\{
    \begin{array}{lr}
      Y_i(\omega) & \mathrm{if}\quad d_\infty(Y_i(\omega),\Lambda(\mu))<R, \\
      \Lambda(\mu) & \mathrm{if}\quad d_\infty(Y_i(\omega),\Lambda(\mu))\geq R.
    \end{array}
    \right.
\end{equation*}
Then, in particular there exists an $N_1\in\mathbb{N}$ such that for all $n\geq N_1$ we have that
\begin{equation}\label{eq7:T:Sturmslln}
d_\infty(S_n,S^{R}_n)<\frac{\epsilon}{20}
\end{equation}
almost surely. Next, we pick a $k\in\mathbb{N}$ such that by Lemma~\ref{L:empiricalMeasureW1} we have
\begin{equation}\label{eq2:T:Sturmslln}
\mathbb{E}W_1(\mu^R,\mu^R_k)<\frac{\epsilon}{5}
\end{equation}
where $\mu^R_k:=\sum_{i=1}^k\frac{1}{k}\delta_{Y^R_i}$. Notice that $\{Y^R_i\}_{i\in\mathbb{N}}$ is an i.i.d. sequence of $\mathbb{P}$-valued random variables with law $\mu^R$. Additionally define
\begin{equation*}
\mu^R_{k,n}:=\sum_{i=1}^k\frac{1}{k}\delta_{Y^R_{nk+i}}.
\end{equation*}
We follow the proof of Theorem~\ref{T:Nodice} and obtain the identities and estimates of \eqref{eq6:T:Nodice},\eqref{eq7:T:Nodice},\eqref{eq8:T:Nodice},\eqref{eq9:T:Nodice} in which $d_\infty(A_1,I)$ is replaced by $z:=\sup\{d_\infty(A,I):A\in\supp(\mu^R)\}$ and $\mathrm{diam}(\supp(\mu))$ by 
$\mathrm{diam}(\supp(\mu^R))$. In particular \eqref{eq6:T:Nodice} now takes the form
\begin{equation}\label{eq6:T:Nodice:T:Sturmslln}
\begin{split}
S^R_{nk}-S^R_{k(n+1)}&+O\left(8\frac{\mathrm{diam}(\supp(\mu^R))^2}{n^2}\right)\left(e^{2z+4\mathrm{diam}(\supp(\mu^R))}\right.\\
&\left.+2\mathrm{diam}(\supp(\mu^R))e^{2z+4\mathrm{diam}(\supp(\mu^R))}\right)\\
&+\sum_{i=1}^{k}\frac{1}{k(n+1)}\log_{S^R_{k(n+1)}}Y_{nk+i}=0,
\end{split}
\end{equation}
and for large $n>0$ we have that the norm of the quantity
\begin{equation*}
\begin{split}
E:=&O\left(8\frac{\mathrm{diam}(\supp(\mu^R))^2}{n^2}\right)\left(e^{2z+4\mathrm{diam}(\supp(\mu^R))}\right.\\
&\left.+2\mathrm{diam}(\supp(\mu^R))e^{2z+4\mathrm{diam}(\supp(\mu^R))}\right)
\end{split}
\end{equation*}
is arbitrarily small. Thus there exists an $\overline{N}_2\in\mathbb{N}$ satisfying also $8\frac{\mathrm{diam}(\supp(\mu^R))^2}{\overline{N}_2^2}<1$, such that by Lemma~\ref{L:resolventAppr} there exists an $\overline{S}^{R,k}_{n}\in\mathbb{P}$ such that
\begin{equation}\label{eq8:T:Nodice:T:Sturmslln}
\frac{1}{n+1}\sum_{i=1}^{k}\frac{1}{k}\log_{S^R_{k(n+1)}}Y_{nk+i}+\log_{S^R_{k(n+1)}}\overline{S}^{R,k}_{n}=0
\end{equation}
and
\begin{equation}\label{eq9:T:Nodice:T:Sturmslln}
d_\infty(S^R_{kn},\overline{S}^{R,k}_{n})\leq O\left(\frac{1}{n^2}\right)
\end{equation}
for all $n\geq \overline{N}_2$. In other words $S^R_{k(n+1)}=J_{\frac{1}{n+1}}^{\mu^R_{k,n}}\overline{S}^{R,k}_{n}$. Now let $S^{R,k}_{n}:=\overline{S}^{R,k}_{N_2}$ for $1\leq n\leq \overline{N}_2$, and $S^{R,k}_{n+1}:=J_{\frac{1}{n+1}}^{\mu^R_{k,n}}S^{R,k}_{n}$ for $n\geq \overline{N}_2$ recursively. Then using the contraction property \eqref{eq:P:ResolventContraction} and \eqref{eq9:T:Nodice:T:Sturmslln} we get
\begin{equation*}
\begin{split}
d_\infty(S^R_{k(n+1)},S^{R,k}_{n+1})&\leq \frac{1}{1+\frac{1}{n+1}}d_\infty(\overline{S}^{R,k}_{n},S^{R,k}_{n})\\
&\leq \frac{1}{1+\frac{1}{n+1}}\left(d_\infty(S^R_{kn},S^{R,k}_{n})+d_\infty(S^R_{kn},\overline{S}^{R,k}_{n})\right)\\
&\leq\left(1-\frac{1}{n+2}\right)\left[d_\infty(S^R_{kn},S^{R,k}_{n})+O\left(\frac{1}{n^2}\right)\right]\\
&=\left(1-\frac{1}{n+2}\right)d_\infty(S^R_{kn},S^{R,k}_{n})+O\left(\frac{1}{(n+2)^2}\right).
\end{split}
\end{equation*}
So combined with Lemma~\ref{L:convrate} we get that there exists an $N_2\in\mathbb{N}$ such that $N_2\geq \overline{N}_2$ and for all $n\geq N_2$ we have that
\begin{equation}\label{eq4:T:Sturmslln}
d_\infty(S^R_{nk}(\omega),S^{R,k}_n(\omega))<\frac{\epsilon}{5}.
\end{equation}
Define the auxiliary sequence $\hat{S}^{R}_n:=S^{R,k}_n$ for $1\leq n\leq \overline{N}_2$, and $\hat{S}^{R}_{n+1}:=J_{\frac{1}{n+1}}^{\mu^R}(\hat{S}^{R}_n)$ for $n\geq \overline{N}_2$ recursively.
Then using Lemma~\ref{L:empiricalIteration} we obtain
\begin{equation*}
d_\infty(S^{R,k}_n(\omega),\hat{S}^{R}_n(\omega))\leq \frac{n-N_2+1}{n}\sum_{i=\overline{N}_2}^{n}\frac{1}{n-N_2+1}W_1(\mu^R,\mu^R_{k,i}(\omega))
\end{equation*}
and by \eqref{eq2:T:Sturmslln}, the $L^1$-strong law of large numbers for the real valued random variable $W_1(\mu^R,\mu^R_{k,1}(\omega))$ yields that there exists an $N_3\in\mathbb{N}$ such that for all $n\geq N_3$ we have that $\sum_{i=\overline{N}_2}^{n}\frac{1}{n-N_2+1}W_1(\mu^R,\mu^R_{k,i})<\frac{\epsilon n}{5(n-N_2+1)}$ almost surely; thus
\begin{equation}\label{eq5:T:Sturmslln}
d_\infty(S^{R,k}_n,\hat{S}^{R}_n)<\frac{\epsilon}{5}
\end{equation}
for all $n\geq N_3$ almost surely. Next, Proposition~\ref{P:Proximal} yields that there exists an $N_4\in\mathbb{N}$ such that for all $n\geq N_4$ we have that
\begin{equation}\label{eq6:T:Sturmslln}
d_\infty(\hat{S}^{R}_n,\Lambda(\mu^R))<\frac{\epsilon}{5}.
\end{equation}

Finally for fixed $k,R$ depending only on $\epsilon$,$\mu$; by combining the estimates \eqref{eq1:T:Sturmslln}, \eqref{eq4:T:Sturmslln}, \eqref{eq5:T:Sturmslln}, \eqref{eq6:T:Sturmslln}, \eqref{eq7:T:Sturmslln} we obtain that for all $n\geq \max\{N_1,N_2,N_3,N_4\}$
\begin{equation}\label{eq8:T:Sturmslln}
\begin{split}
d_\infty(S_{nk},\Lambda(\mu))&\leq d_\infty(S_{nk},S^{R}_{nk})+d_\infty(S^R_{nk},S^{R,k}_n)+d_\infty(S^{R,k}_n,\hat{S}^{R}_n)\\
&\quad+d_\infty(\hat{S}^{R}_n,\Lambda(\mu^R))+d_\infty(\Lambda(\mu^R),\Lambda(\mu))<\frac{7}{10}\epsilon
\end{split}
\end{equation}
almost surely. What remains to be seen is that $d_\infty(S_{nk},S_{nk+i})$ is also proportionally small for any $1\leq i\leq k-1$ and large enough $n\geq \max\{N_1,N_2,N_3,N_4\}$ almost surely. By \eqref{eq7:T:Sturmslln} it is enough to estimate $d_\infty(S^{R}_{nk},S^{R}_{nk+i})$, since $d_\infty(S_{nk},S_{nk+i})\leq d_\infty(S_{nk},S^R_{nk})+d_\infty(S^{R}_{nk},S^{R}_{nk+i})+d_\infty(S^R_{nk+i},S_{nk+i})$. We claim that
\begin{equation}\label{eq9:T:Sturmslln}
d_\infty(S^{R}_{n},\Lambda(\mu))\leq R
\end{equation}
for all $n\in\mathbb{N}$. Indeed $d_\infty(Y^{R}_{n},\Lambda(\mu))\leq R$ and $d_\infty(S^{R}_{1},\Lambda(\mu))\leq R$ by the definition of $Y^{R}_{n}$ and recursively we have
\begin{equation*}
\begin{split}
d_\infty(S^{R}_{n+1},\Lambda(\mu))&=d_\infty(S^{R}_{n}\#_{\frac{1}{n+1}}Y^{R}_{n+1},\Lambda(\mu))\\
&\leq\left(1-\frac{1}{n+1}\right)d_\infty(S^{R}_{n},\Lambda(\mu))+\frac{1}{n+1}d_\infty(Y^{R}_{n+1},\Lambda(\mu))\\
&\leq\frac{n}{n+1}d_\infty(S^{R}_{n},\Lambda(\mu))+\frac{1}{n+1}R,
\end{split}
\end{equation*}
where we used \eqref{eq0:P:KarcherW1contracts} to obtain the first estimate. Thus using \eqref{eq9:T:Sturmslln} we get
\begin{equation*}
\begin{split}
d_\infty(S^{R}_{nk},S^{R}_{nk+i})&\leq \sum_{j=0}^{i-1}d_\infty(S^{R}_{nk+j},S^{R}_{nk+j+1})\\
&=\sum_{j=0}^{i-1}d_\infty(S^{R}_{nk+j},S^{R}_{nk+j}\#_{\frac{1}{nk+j+1}}Y^{R}_{nk+j+1})\\
&=\sum_{j=0}^{i-1}\frac{1}{nk+j+1}d_\infty(S^{R}_{nk+j},Y^{R}_{nk+j+1})\\
&\leq\sum_{j=0}^{i-1}\frac{1}{nk}d_\infty(S^{R}_{nk+j},Y^{R}_{nk+j+1})\\
&\leq\sum_{j=0}^{i-1}\frac{1}{nk}[d_\infty(S^{R}_{nk+j},\Lambda(\mu))+d_\infty(Y^{R}_{nk+j+1},\Lambda(\mu))]\\
&\leq\sum_{j=0}^{i-1}\frac{1}{nk}2R\leq\sum_{j=0}^{k-1}\frac{1}{nk}2R=\frac{2R}{n}.
\end{split}
\end{equation*}
This proves that almost surely $d_\infty(S_{nk},S_{nk+i})<2\frac{\epsilon}{20}+\frac{2R}{n}$, so if additionally $n\geq \frac{10R}{\epsilon}$, then $d_\infty(S_{nk},S_{nk+i})<\frac{3}{10}\epsilon$ almost surely for all $1\leq i\leq k-1$. Thus combined with \eqref{eq8:T:Sturmslln} yields $d_\infty(S_{nk+i},\Lambda(\mu))<\epsilon$ for all $1\leq i\leq k-1$ and $n\geq \max\{N_1,N_2,N_3,N_4,\frac{10R}{\epsilon}\}$ almost surely.
\end{proof}

\subsection*{Acknowledgments}
The work of Y.~Lim was supported by the National Research Foundation of Korea (NRF) grant funded by the Korea government(MEST) No.2015R1A3A2031159 and 2016R1A5A1008055.

The work of M.~P\'alfia was supported by the National Research Foundation of Korea (NRF) grants funded by the Korea government(MEST) No.2015R1A3A2031159, No.2016R1C1B1011972 and No.2019R1C1C1006405, Ministry of Human Capacities, Hungary grant 20391-3/2018/FEKUSTRAT and the Hungarian National Research, Development and Innovation Office NKFIH, Grant No. FK128972.

\end{document}